\documentclass[reqno,11pt]{amsart}
\usepackage{amsmath,amssymb,mathrsfs,amsthm,amsfonts,comment}
\usepackage[inline]{enumitem}
\usepackage[usenames,dvipsnames]{xcolor}
\usepackage{hyperref}
\hypersetup{%
	colorlinks=true, linkcolor=blue,
	citecolor=ForestGreen
}
\usepackage{bbm}
\usepackage[paper=letterpaper,margin=1in]{geometry}

\newtheorem{theorem}{Theorem}[section]

\newtheorem{lemma}[theorem]{Lemma}
\newtheorem{proposition}[theorem]{Proposition}

\theoremstyle{definition}
\newtheorem{definition}[theorem]{Definition}
\newtheorem{remark}[theorem]{Remark}
\newtheorem{innercustomprop}{Proposition}
\newenvironment{customprop}[1]
  {\renewcommand\theinnercustomprop{#1}\innercustomprop}%
  {\endinnercustomprop}

\numberwithin{equation}{section}
\allowdisplaybreaks
\usepackage{acronym}

\theoremstyle{definition}
\numberwithin{equation}{section}
\allowdisplaybreaks
\usepackage{acronym,xcolor}

\acrodef{KPZ}{Kardar--Parisi--Zhang}
\acrodef{SHE}{Stochastic Heat Equation}
\acrodef{LDP}{Large Deviation Principle}

\renewcommand{\Pr}{\mathbb{P}}
\newcommand{\Ex}{\mathbb{E}}
\renewcommand{\d}{\mathrm{d}}	

\newcommand{\nw}{\mathbf{nw}}
\newcommand{\hyph}{\mathbf{Hyp}_{1}(\mathfrak{p})}
\newcommand{\hypz}{\mathbf{Hyp}_{2}(\mathfrak{q})}

\newcommand{\norm}[1]{\big[#1\big]}

\newcommand{\Con}{\mathrm{C}} 

\newcommand{\m}{\mathsf}
\newcommand{\R}{\mathbb{R}} 
\newcommand{\Z}{\mathbb{Z}} 

\newcommand{\e}{\varepsilon}
\newcommand{\h}{\mathfrak{h}}

\newcommand{\calH}{\mathcal{H}}

\newcommand{\calZ}{\mathcal{Z}}
\newcommand{\calV}{\mathcal{V}}

\newcommand{\B}{\mathfrak{B}}

\renewcommand{\hat}{\widehat}
\newcommand{\til}{\widetilde}

\renewcommand{\v}{\mathcal{V}}

\usepackage{graphicx}

\title[Temporal Increments of the KPZ equation]{Temporal Increments of the KPZ equation with general initial data}

\author[S.\ Das]{Sayan Das}
\address{S.\ Das,
	Department of Mathematics, Columbia University,
	\newline\hphantom{\quad \ \ S. Das}
	2990 Broadway, New York, NY 10027 USA
}
\email{sayan.das@columbia.edu}


\begin{document}

\begin{abstract} We consider the Cole-Hopf solution of the (1+1)-dimensional KPZ equation $\mathcal{H}^f(t,x)$ started with initial data $f$. In this article, we study the sample path properties of the \textit{KPZ temporal process} $\calH_t^f:=\mathcal{H}^f(t,0)$. We show that for a very large class of initial data which includes any deterministic continuous initial data that grows slower than parabola $(f(x)\ll 1+x^2$) as well as narrow wedge and stationary initial data, temporal increments of $\calH_t^f$ are well approximated by increments of a fractional Brownian motion of Hurst parameter $\frac14$. As a consequence, we obtain several sample path properties of KPZ temporal properties including variation of the temporal process, law of iterated logarithms, modulus of continuity and Hausdorff dimensions of the set of points with exceptionally large increments.
\end{abstract}

\maketitle

\section{Introduction}
In this article we study the local sample path properties of the \ac{KPZ} equation. KPZ equation is a stochastic partial differential equation formally given by
\begin{align*}
	\partial_t\calH=\tfrac12\partial_{xx}\calH+\tfrac12(\partial_x\mathcal{H})^2+\xi, \qquad \calH=\calH(t,x), \quad (t,x)\in [0,\infty)\times\mathbb{R}
\end{align*}
where $\xi=\xi(t,x)$ is the Gaussian space-time white noise. Introduced in \cite{kpz}, the KPZ equation is a popular model for interfaces of random growth. However, the equation is ill-posed due to the presence of the non-linear term $(\partial_x\calH)^2$. To make sense of the equation one considers $\calZ:=e^{\calH}$ which formally solves the \ac{SHE} with multiplicative noise:
\begin{align}\label{she}
	\partial_t\calZ=\tfrac{1}{2}\partial_{xx}\calZ+\xi\calZ,  \qquad \calZ=\calZ(t,x), \quad (t,x)\in [0,\infty)\times\mathbb{R}.
\end{align}
The \ac{SHE} is known to be well-posed and has a well-developed solution theory based on It\^{o} integral and chaos expansion~\cite{walsh1986,bertini1995stochastic,quastel2011introduction,corwin2012kardar}. In particular, {\eqref{she}  with the initial data $\calZ(0,x)=g(x)$ can be rewritten in Duhamel form as}
\begin{align}\label{fkdc}
	\calZ(t,x) = \int_{\R} p_t(x-y)g(y)\,\d y + \int_0^t\int_{\R} p_{t-s}(x-y) \calZ(s,y) \xi(\d s, \d y).
\end{align}
Here $ p_t(x) := (2\pi t)^{-1/2} \exp(-x^2/(2t))$ is the standard heat kernel. {There exists a unique mild solution for the above equation  with the initial data $\calZ(0,x)=g(x)$ independent of the noise $\xi$ and satisfying}
\begin{align*}
    {\Ex \left[\sup_{A\subset [-n,n]} \int_{A} g(x)dx\right] \le \Con e^{\Con n}}
\end{align*}
{for some constant $C>0$ and for all $n\in \Z_{>0}$}. 

For a function-valued initial data $\calZ(0,x)\ge 0$ that is not identically zero, it is known
$\calZ(t,x)>0$ for all $t > 0$ and $x\in \R$ almost surely \cite{mueller1991support}. Thus $ \calH := \log \calZ$ is well defined and is termed as the Cole-Hopf solution of the KPZ equation. This notion of solution coincides with other notions of solutions, such as  regularity structures \cite{hairer2013solving,hairer2014theory}, paracontrolled distributions \cite{gubinelli2015paracontrolled,gubinelli2017kpz}, and energy solutions \cite{gonccalves2014nonlinear,gubinelli2018energy}, under suitable assumptions. 

An often considered initial data is to start the \ac{SHE} from a Dirac delta at the origin, i.e., 
\begin{align}\label{deltaic}
	\calZ^{\nw}(0,x):=\delta_0(x)
\end{align}	
which is referred to as
the narrow wedge initial data for the KPZ equation $\calH^{\nw}$. For such an initial data, \cite{flo14} established the positivity for $\calZ^{\nw}(t,x)$ so that the
Cole-Hopf solution $\calH^{\nw} := \log \calZ^{\nw}$ is well-defined.\\

In this paper, we consider a general class of initial data for the KPZ equation which we now describe.

\begin{definition}\label{hyp1} Fix $\mathfrak{p}=(\theta, \delta, \lambda,\kappa,M) \in (0,\infty)^5$.  The $\hyph$ class is a class of possible initial data for the KPZ equation consisting of:
	
	\begin{enumerate}
		\item \label{acond} Narrow wedge initial data corresponding to delta initial data \eqref{deltaic} for \ac{SHE}.
		
		\item \label{bcond} Brownian initial data: The two sided Brownian motion $B(x)$ independent from the space-time white noise $\xi$.
		
		\item \label{ccond} Measurable deterministic functions $f: \R\to \R\cup \{-\infty\}$ satisfying the following two conditions:
		\begin{enumerate}
			\item \label{ci} {$f(x)\le \lambda(1+|x|)^{2-\delta}$},
			
			\item \label{cii} there exist a subinterval $\mathcal{I} \subset [-M,M]$ with $|\mathcal{I}|=\theta$ such that $f(y)\ge -\kappa$ for all $y\in \mathcal{I}$. 
		\end{enumerate}
	\end{enumerate}		
\end{definition}

We consider the \ac{SHE} $\calZ^f(t,x)$ \eqref{she} started from initial data $e^f$ where $f$ belong to the above class and define the KPZ equation with initial data $f$ via the Cole-Hopf transform $\calH^f(t,x):=\log \calZ^f(t,x)$. We consider the \textit{KPZ temporal process} $\mathcal{H}^f(t,0)$ as a function of time with spatial point fixed at $0$. We use the abbreviated notation $\mathcal{Z}_t^f(x):=\mathcal{Z}^f(t,x)$, $\mathcal{Z}_t^f:=\mathcal{Z}_t^f(0)$, $\mathcal{H}_t^f(x):=\mathcal{H}^f(t,x)$, and $\mathcal{H}_t^f:=\mathcal{H}_t^f(0)$. Our main result, which is stated below, shows that the KPZ temporal process $\mathcal{H}_t^f$ can be approximated by a fractional Brownian motion at a local level uniformly over compact sets of $(0,\infty)$.

\begin{theorem}\label{t:main} Fix $\mathfrak{p}\in (0,\infty)^5$, $k\ge 2$, $T>1$, and $\beta\in (0,\frac25)$. Consider any $f$ from the $\hyph$ class defined in Definition \ref{hyp1}. One can construct a probability space consisting of the KPZ temporal process $\mathcal{H}_t^f$ started with $f$, a fractional Brownian motion $\B_t$ with Hurst parameter $\frac14$ (defined in Definition \ref{fbm}), such that the process: $\{\calH_t^f-(2/\pi)^{\frac14}\B_t: t\in [T^{-1},T]\}$ is locally H\"older continuous with exponent $\beta$ and furthermore  $$\Ex\left[\norm{\calH_t^f-(2/\pi)^{\frac14}\B_t}_{C^{0,\beta}[T^{-1},T]}^k\right]<\infty.$$
\end{theorem}
Here $\norm{f}_{C^{0,\beta}[a,b]}$ denotes the H\"older coefficient of $f$ which is defined as 
\begin{align*}
	\norm{f}_{C^{0,\beta}[a,b]}:=\sup_{x\neq y\in [a,b]}\frac{|f(x)-f(y)|}{|x-y|^{\beta}}.
\end{align*}

\begin{remark}  In plain words, Theorem \ref{t:main} shows that one can put $\mathcal{H}_t^f$, $\B_t$ in a common probability space such that ${\mathcal{H}_t^f-(2/\pi)^{\frac14}\B_t}$ on any compact interval of $(0,\infty)$ is H\"older continuous with H\"older exponent less than $\frac25$. Since $\B_t$ is itself a H\"older continuous process with  H\"older exponent less than $\frac14$, Theorem \ref{t:main} entails that $\mathcal{H}_t^f$ is almost surely H\"older continuous with H\"older exponent being less than $\frac14$ as well. This recovers the result from  \cite[Theorem 2.2]{bertini1995stochastic} (see also \cite{walsh1986}) which shows that the solution to the \ac{SHE} (for a wide class of functional initial data including delta initial data) as a process in time is almost surely H\"older continuous with H\"older exponents less than $\frac14$.  
\end{remark}

The sharp comparison in Theorem \ref{t:main} allows us to obtain several sample path properties of the KPZ temporal process $\mathcal{H}_t^f$ as a consequence.  To state the results, we recall the notion of $\alpha$-variation of a function. Given a function $g :[s,t]\to \mathbb{R}$ and $\alpha>0$, define the $\alpha$-variation of $g$ at scale $\e$ by
\begin{align*}
	\v_{\alpha,\e}(g):=\sum_{u\in[s+\e,t]\cap \e \Z} |g(u)-g(u-\e)|^{\alpha}.
\end{align*}
With this definition in place, we are now ready to state the theorem regarding the temporal increments of the KPZ equation for initial data drawn from $\hyph$.
\begin{theorem}\label{t:cor}  Fix $\mathfrak{p}\in (0,\infty)^5$ and consider any $f$ from the $\hyph$ class defined in Definition \ref{hyp1}. Consider the KPZ temporal process $\calH_t^f$ started from $f$. The following holds:
	\begin{enumerate}
		\item Temporal increments of $\calH^f$ are asymptotically normal. For $0<t_1<t_2<\cdots<t_k<\infty$, as $\e \downarrow 0$,
		\begin{align*}
			\big((\pi/2)^{\frac14}\e^{-\frac14}[\calH_{t_i+\e}^f-\calH_{t_i}^f]\big)_{i=1}^k \stackrel{d}{\to} \mathcal{N}_k(\mathbf{0},\mathbf{I}_k).
		\end{align*}
		where $\mathcal{N}_k(\mathbf{0},\mathbf{I}_k)$ is a standard Gaussian $k$-vector with mean $\mathbf{0}$ and variance-covariance matrix $\mathbf{I}_k$. 
		
		\item  $\calH^f$ has quartic variation in the sense that for any $[s,t] \subset (0,\infty)$, we have
		\begin{align*}
			\v_{4,\e}(\calH^f\mid_{[s,t]}) \stackrel{p}{\to} \tfrac{6}{\pi}(t-s).
		\end{align*}
		\item For any $t>0$ with probability $1$ we have
		\begin{align*}
			\limsup_{\e\downarrow 0} \frac{\calH_{t+\e}^f-\calH_t^f}{\e^{1/4}\sqrt{\log\log(1/\e)}}= \left(\frac{8}{\pi}\right)^{1/4}.
		\end{align*}
		\item With probability $1$ we have
		\begin{align*}
			\limsup_{\e\downarrow 0}\frac{1}{\e^{\frac14}\sqrt{\log (1/\e)}}\sup_{\substack{1\le s<t\le 2\\ |t-s|\le \e}} |\calH_t^f-\calH_s^f|=\left(\frac{8}{\pi}\right)^{1/4}.
		\end{align*}
		\item Set
		\begin{align*}
			E(\alpha):=\left\{t\in [1,2]: \limsup_{\e\downarrow 0} \frac{|\calH_{t+\e}^f-\calH_t^f|}{\e^{\frac14}\sqrt{\log (1/\e)}} \ge \alpha\left(\frac{8}{\pi}\right)^{\frac14} \right\}.
		\end{align*}
		We have $\dim_H E(\alpha)=1-\alpha^2$ almost surely where $\dim_H$ denotes the Hausdorff dimension.
	\end{enumerate}
\end{theorem}

\subsection{Proof Idea}\label{sec:pfidea} In this section, we illustrate the main ideas that go behind the proof of Theorem \ref{t:main}. The proof strategy is to first establish a similar temporal approximation at the level of the Stochastic Heat Equation and then update it to the KPZ equation level using tail estimates for the KPZ equation. Our approximation for the \ac{SHE} holds for a wider class of initial data, namely the $\hypz$ class (see Definition \ref{hyp} for details). \\

Let us consider \ac{SHE} $\calZ^f(t,x)$ started from initial data $f$ drawn from $\hypz$ class (see Definition \ref{hyp}). We first study the temporal differences of the \ac{SHE}. Using \eqref{fkdc} for each $t,\e>0$, one has
\begin{equation}\label{e:ztemp}
	\begin{aligned}
		\calZ^f(t+\e,0)-\calZ^f(t,0) & = \int_{\R} [p_{t+\e}(y)-p_t(y)]e^{f(y)}\d y \\ & \hspace{1cm}+\int_{t}^{t+\e}\int_{\R} p_{t+\e-s}(y)\calZ^f(s,y)\xi(\d s,\d y) \\ & \hspace{2cm}+\int_{0}^{t}\int_{\R} [p_{t+\e-s}(y)-p_{t-s}(y)]\calZ^f(s,y)\xi(\d s,\d y).
	\end{aligned}
\end{equation}
Assuming $\e$ is small enough, we seek to find a good approximation for the r.h.s.~of above equation. We now give a heuristic idea of how to arrive at a good approximation. We remark that the following approach was outlined in \cite{davar}. \\

For the class of initial data considered here, the first term on the r.h.s.~of \eqref{e:ztemp} can be shown to be of the order $O(\e)$ which turns out to be much more regular than the other two terms. Thus this term can be ignored.

Let us consider the third term in r.h.s.~of \eqref{e:ztemp}. Note that the heat kernel $p_t(x)$ is singular near $t=0$ and for small enough $t$,  $p_t(y)\d y \approx \delta_0(\d y)$. As long as $\calZ^f(s,y)$ is nice enough, one expects $$\int_{0}^{t'}\int_{\R} [p_{t+\e-s}(y)-p_{t-s}(y)]\calZ^f(s,y)\xi(\d s,\d y)$$ is regular enough for $t'<t$. Hence the main contribution in the third term on r.h.s.~of \eqref{e:ztemp} comes when $s$ is close to $t$. For $s$ close to $t$, the essential contribution of $[p_{t+\e-s}(y)-p_{t-s}(y)]$ comes from $y=0$.  Similarly in the second term, as $s$ is already close to $t$, again the contribution of $p_{t+\e-s}(y)$ comes when $y$ is close to zero. Thus naively we may replace $\calZ^f(s,y)$ by $\calZ^f(t,0)$ in the second and third terms. We then expect the quantity on the r.h.s.~of \eqref{e:ztemp} is close to
\begin{equation*}
	\begin{aligned}
		\calZ^f(t,0) \cdot \left[\int_{0}^{t}\int_{\R} [p_{t+\e-s}(y)-p_{t-s}(y)]\xi(\d s,\d y)+\int_{t}^{t+\e}\int_{\R} p_{t+\e-s}(y)\xi(\d s,\d y)\right],
	\end{aligned}
\end{equation*} 
which is same as
\begin{equation*}
	\begin{aligned}
		\calZ^f(t,0) \cdot \left[\mathcal{V}_{t+\e}(0)-\mathcal{V}_t(0)\right],
	\end{aligned}
\end{equation*}
where
\begin{align*}
	\calV_t(x):=\int_{(0,t)\times \R} p_{t-s}(y-x)\xi(\d s,\d y)
\end{align*}
with the same underlying space-time white noise $\xi(t,x)$.  $\calV_t(x)$ solves the linear stochastic heat equation $\partial_t\calV=\frac12\partial_{xx}\calV+\xi$ with $\calV_0(x)\equiv 0$. By Proposition 2.2 in \cite{davar} ($\alpha=2$ in their statement), one can find a fractional Brownian motion $\B_t$ with Hurst parameter $\frac14$ and another process $R_t$ on the same probability space as $\calV_t$ such that  $$\calV_t=(2/\pi)^{\frac14}\B_t+R_t.$$ 
Finally, the process $R_t$ is known to be much more regular and can be ignored. Thus the above approach leads to the following approximation theorem. 

\begin{theorem}\label{t:comp} Fix $\mathfrak{q}=(\mathfrak{q}_k)_{k=1}^{\infty} \in ((0,\infty)\times (0,1])^{\mathbb{N}}$ and fix any $f\in \hypz$ defined in Definition \ref{hyp}. Consider the \ac{SHE} defined in \eqref{she} with initial data $e^f$. Fix $T>1$ and $k\ge 2$ and $\beta\in (0,\frac25)$. There exists a fractional Brownian motion $\B_t$ with Hurst parameter $\frac14$ on the same probability space and constants $\Con=\Con(\beta,\mathfrak{q}_k,k,T)>0$ and $\e_0=\e_0(\beta,\mathfrak{q}_k,k,T)>0$ such that for all $\e\in (0,\e_0)$ we have
	\begin{align}\label{mombd}
		\sup_{t\in [T^{-1},T]}\Ex\left[|\calZ_{t+\e}^f-\calZ_t^f-(2/\pi)^{1/4}\calZ_t^f(\B_{t+\e}-\B_t)|^k\right] \le \Con\e^{k\beta}.
	\end{align}
\end{theorem}
We mention that this result appears in \cite{davar} for flat initial data $\calZ(0,x)\equiv \mbox{Const}$. Once Theorem \ref{t:comp} is established, a similar approximation theorem for the KPZ equation is established using sharp tail estimates for the KPZ equation that holds for general initial data. These sharp tail estimates are proven in \cite{kpzgen} and holds for all initial data in $\hypz$ class.

To make the above argument rigorous, one requires uniform moment estimates as well as spatial and temporal regularity moment estimates for the SHE. Such estimates are well known for flat initial data (Lemma 3.1 in \cite{davar}). One of the key technical achievements of this paper  is to extend these moment estimates to general initial data. Given any $T>1$,  we are interested in estimating moments of three types:

\begin{itemize}
	\itemsep0.7em
	\item $\Ex[|\calZ_t^f(y)|^k]$ for all $t\in (0,T]$ and $y\in \R.$
	\item $\Ex[|\calZ_t^f(y)-\calZ_t^f(0)|^k]$ for all $t\in [T^{-1},T]$ and $|y|$ small enough.
	\item $\Ex[|\calZ_s^f(0)-\calZ_t^f(0)|^k]$ for all $s,t\in [T^{-1},T]$.
\end{itemize}

For the first point, we show that the logarithm of the moments is of the form $(1+|y|)^{2-\delta}$.  For the next two points, we produce regularity estimates. We show that the expressions in the second and the third points above are bounded by $|y|^{k\theta}$ (for any $\theta\in (0,\frac12)$) and $|s-t|^{k/4}$ respectively. 

We first prove the uniform moment estimate and spatial regularity estimate for narrow wedge data using the moment estimates and tail estimates from \cite{kpzgen} and \cite{kpztime} respectively. We then upgrade it to general initial data via the \textit{convolution formula} that relates the spatial process of $\calZ^{f}(t,\cdot)$ with that of $\calZ^{\nw,y}(t,\cdot)$ started with initial data $\delta_y(x)$.

\begin{proposition}[Convolution formula]\label{p:conv} For any functional initial data in $\hypz$, one has
	$$\calZ_t^f(x)=\int_{\R} \calZ_t^{\nw,y}(x)e^{f(y)}\d y,$$
	where $\calZ_t^{\nw,y}(x)$ denotes the \ac{SHE} started with initial $\delta_{y}(x)$. Furthermore for each fixed $y$ and $t>0$, 
	\begin{align}\label{idenp}
		\calZ_t^{\nw,y}(x)\stackrel{d}{=}\calZ_t^{\nw}(x-y)
	\end{align} 
	as processes in $x$.
\end{proposition}

The proof of the above proposition is immediate from the chaos expansion for the \ac{SHE} and hence skipped. {Owing to Proposition \ref{p:conv}, we see that the spatial process $\calZ_t^f(\cdot)$ is equal in distribution to $\int_{\R} \calZ_t^{\nw}(\cdot-y)e^{f(y)}\d y$.  The latter quantity is a function of $\calZ_t^{\nw}$ and the initial data $f$ only. Thus, inserting spatial regularity estimates for $\calZ_t^{\nw}$ in the second term leads to spatial regularity estimates for $\calZ_t^f$ (see Section \ref{sec:spmom} for details)}. 

  {Although the temporal difference $\calZ_s^f-\calZ_t^f$ can be written as $$\int_{\R} (\calZ_s^{\nw,y}(0)-\calZ_t^{\nw,y}(0))e^{f(y)}\d y,$$ the quantity $\calZ_s^{\nw,y}(0)-\calZ_t^{\nw,y}(0)$ can not be viewed as a function of a \textit{single} narrow wedge solution of SHE (as \eqref{idenp} holds for each fixed $t>0$ only). Thus the convolution formula alone} is not very helpful in obtaining temporal regularity estimates. To derive the temporal regularity estimate, we rely on the chaos expansion instead and use the uniform moment estimates for general initial data.

\subsection{Previous Works} Our main results about the KPZ temporal process fit into the broader endeavor of understanding the local structure and multifractal properties of stochastic systems. Studying fractal properties of stochastic partial differential equations (SPDEs) is an active area of research for the last five-six decades. The main focus of a vast majority of such works resided in the study of the large peaks of the SPDEs with multiplicative noise \cite{GM90,CM94,bertini1995stochastic,HHNT,FK09,CJKS,CD15,BC16, Ch17,CHN19}. The growth of the large peaks is also connected to the intermittency phenomenon of SPDEs. We refer to \cite{bertini1995stochastic, CM94, khoshnevisan2014analysis} for a detailed discussion. 

Local structure of the \ac{SHE} and the \ac{KPZ} were explored in  \cite{davar,foondun,hk}. In \cite{davar}, the authors considered a generalization of the \ac{SHE} in \eqref{she}: 
\begin{align}\label{genshe}
	\partial_t\calZ=\tfrac12\Delta_{\alpha/2}\calZ+\sigma(\calZ) \xi, \qquad \calZ(0,x)\equiv 1,
\end{align}
where $\sigma$ is any continuous Lipschitz function and $\alpha\in (1,2]$. Here $\Delta_{\alpha/2}$ denotes the fractional Laplace operator acting on the spatial variable and it is defined as $$\hat{(\Delta_{\alpha/2}\psi)}(\xi):=-|\xi|^{\alpha}\hat{\psi}(\xi).$$ It was shown that locally $t\mapsto \calZ(t,x)$ looks like a fixed multiple of fractional Brownian motion of Hurst parameter $H=\frac12-\frac1{2\alpha}.$ 	The approximation in \cite{davar} is strong enough to deduce several fractal properties $\calZ$. For $\alpha=2$, \cite{hk} studied the short-time peaks of $t\mapsto \calZ(t,x)$. In particular, they studied the Hausdorff and Box dimension of the set
$$\mathcal{U}_c:=\left\{x\in [0,1] \mid \limsup_{\e\downarrow 0} \frac{\calZ(\e,x)-1}{\e^{\frac14}\sqrt{\log(1/\e)}}\ge c\right\},$$
and concluded that high local oscillations of $(t,x)\mapsto \calZ(t,x)$ are multifractal. A similar study was conducted for the spatial increments of the random function the solution of \eqref{genshe} in \cite{foondun}. One of the corollaries of their result is that the spatial process of the KPZ equation for flat initial data is locally Brownian, thus complementing the works of \cite{hairer2013solving,quastel2011local}.  

Long-time variants of the above results were conjectured in the physics literature decades ago and were established rigorously in \cite{KKX17,KKX18}.  \cite{KKX17} investigated the fractal properties of the \ac{SHE} started from the flat initial data and established the multifractal nature of the spatial process of the KPZ equation.  
The same set of authors in a later paper \cite{KKX18} showed that there are infinitely many different stretch scale (in the spatial direction) and time scale such that for any given stretch and time scale, the peaks of the spatio-temporal process of the stochastic heat equation attain non-trivial macroscopic Hausdorff dimensions.  Fractal properties of the peaks of the KPZ temporal process under $1:2:3$ scaling were investigated in \cite{dg} and more recently the gaps of the tall peaks of \ac{SHE} were addressed in \cite{gy21}.

We close this section by reviewing some of the recent works related to the fractal geometry of the KPZ fixed point. The KPZ equation is one of the prototypical models in the KPZ universality class
which contains a large collection of random growth models that are believed to exhibit certain common features such as a universal scaling exponent of $1/3$ and certain universal non-Gaussian large-time fluctuations. We refer to \cite{corwin2012kardar} for an excellent survey on this topic. All such models in the KPZ universality class are conjectured to converge to a universal scaling limit, the KPZ fixed point $\h(t,x)$ constructed in \cite{dov,mqr}. Recently \cite{qs20, vir20} announced proofs of convergence of the KPZ equation (with narrow wedge initial data) under $1:2:3$ scaling to the KPZ fixed point $\h(t,x)$ (with narrow wedge initial data).  Fractal geometry of the KPZ fixed point related to geodesic geometry was studied in \cite{bgh1,bgh2,bg21} (see also \cite{g21} for a survey)  and related to fixed point maximizers were studied in \cite{chhm21,d22}. In the context of local structure, it was recently shown in \cite{dauvergne2020three} that KPZ fixed point has cubic variation in a similar spirit as Theorem \ref{t:cor} (2). The questions about short time and long time law of iterated logarithms for KPZ fixed point are investigated in \cite{dgl}. We mention that KPZ fixed point is not known to be associated with any SPDE and hence the ideas to prove any fractal property for the fixed point are quite different than the ones presented here.

\subsection*{Organization} The rest of the paper is organized as follows. In Section \ref{sec:mom}, we prove several moment estimates for \ac{SHE}. Based on it, we prove Theorem \ref{t:comp} in Section \ref{sec:fbm}. Proof of Theorems \ref{t:main} and \ref{t:cor} are then completed in Section \ref{sec:pf}.   

\subsection*{Acknowledgements} We thank Ivan Corwin for numerous discussions, encouragement, and comments on an earlier draft of the paper. We thank Promit Ghosal, and Kevin Yang for several useful discussions. We thank Weitao Zhu for feedback and {Yier Lin for pointing out a mistake in an earlier draft of the paper}. We thank the anonymous referees for
their careful reading and useful comments on improving our manuscript. The author acknowledges support from NSF DMS-1928930 during their participation in the program ``Universality and Integrability in Random Matrix Theory and Interacting Particle Systems'' hosted by the Mathematical Sciences Research Institute in Berkeley, California during the Fall semester of 2021.


\section{Moment estimates for Stochastic Heat Equation} \label{sec:mom}

Before commencing with the analysis, let us mention a few pieces of notation that we will  use throughout this paper. {Given a function (possibly random) $f:\R \to \R\cup \{-\infty\}$, we use $\calZ^f(t,x)$ or $\calZ_t^f(x)$ to denote the solution to the \ac{SHE} started from the initial data $e^f$. For narrow wedge initial data (i.e., $\delta_0(x)$ initial data for the \ac{SHE}), we denote it as $\calZ^{\nw}(t,x)$ or $\calZ^{\nw}_t(x)$}. When $x=0$, we will often use the shorthand notations $\calZ_t^f:=\calZ_t^{f}(0)$ and $\calZ_t^{\nw}:=\calZ_t^{\nw}(0)$.  We use analogous notations for the KPZ equation as well. We use $\Con = \Con(x, y, z, \ldots) > 0$ to denote a generic deterministic positive finite
constant that may change from line to line, but dependent on the designated variables $x,y,z,\ldots$.

\medskip

In this section, we prove moment estimates for the \ac{SHE} \eqref{she}, for a wide class of initial data which we now introduce. 
\begin{definition}\label{hyp} Fix any $\mathfrak{q}= (\mathfrak{q}_k)_{k=1}^{\infty}=(\lambda_k,\delta_k)_{k=1}^{\infty} \in ((0,\infty)\times (0,1])^{\mathbb{N}}$.   The $\hypz$ class is a class of possible initial data for the KPZ equation consisting of:
	
	\begin{enumerate}
		\item \label{acond1} Narrow wedge initial data corresponding to delta initial data \eqref{deltaic} for \ac{SHE}.
		
		\item \label{bcond1} Consider any measurable (possibly random) function $f:\R \to \R\cup \{-\infty\}$.  For each $k\in \Z_{>0}$, we set \begin{align}\label{mkf}
			{M_k^f(x):=
				\Ex[e^{kf(x)}]},
		\end{align}
{with the convention that $\Ex[e^{k f(x)}]=e^{kf(x)}$ when $f$ is non-random}.	$f$ belongs to the class $\hypz$  if for each $k\in \Z_{>0}$ and for all $x\in \R$ we have
		\begin{align}\label{mkcond}
			\log M_k^f(x) \le \lambda_k(1+|x|^{2-\delta_k}).
		\end{align}
	\end{enumerate}		
\end{definition}
This class of initial data is much larger than the $\hyph$ class introduced in Definition \ref{hyp1}. We record this fact in the following lemma. 

\begin{lemma}\label{gid} Given $\mathfrak{p}\in (0,\infty)^5$, there exists $\mathfrak{q} \in ((0,\infty)\times (0,1])^{\mathbb{N}}$ such that
	$$\hyph \subset \hypz.$$
\end{lemma}
\begin{proof}
	Fix any $\mathfrak{p}=(\theta,\delta,\lambda,\kappa,M)\in (0,\infty)^5$. We set $\lambda_k=\max\{2k\lambda,\frac{k^2}2\}$, $\delta_k=\min(1,\delta)$. Consider $\mathfrak{q}:=(\lambda_k,\delta_k)_{k=1}^{\infty}\in ((0,\infty)\times (0,1])^{\mathbb{N}}$. Clearly for any deterministic function $f\in \hyph$, following the condition (3) in Definition \ref{hyp1} and \eqref{mkf}, we have $\log M_k^f(x)=kf(x) \le k\lambda(1+|x|)^{2-\delta}$. We have
 $$(1+|x|)^{2-\delta} \le (1+|x|)^{\max(1,2-\delta)} \le 2\cdot[1+|x|^{2-\delta_k}].$$
Thus $f\in \hypz$. For Brownian initial data as $\Ex[e^{kB(x)}]=e^{k^2|x|/2}$, we have ${\log} M_k^f(x) \le \frac12k^2|x| \le \lambda_k(1+|x|^{2-\delta_k})$. Hence $\hyph \subset \hypz$.
\end{proof}	

The rest of this section is devoted to proving several kinds of moment estimates for the \ac{SHE}. In Section \ref{sec:umom}, we establish uniform moment estimates in Proposition \ref{umomf}.  In Section \ref{sec:spmom} and Section \ref{sec:tpmom}, we produce spatial and temporal regularity estimates in Proposition \ref{spmomf} and Proposition \ref{tpmomf} respectively.  For Proposition \ref{umomf} and Proposition \ref{spmomf}, we first prove our estimates for the narrow wedge case in Propositions \ref{umom} and \ref{spmom} respectively. Then we upgrade it to general initial data by convolution formula (Proposition \ref{p:conv}). 

\subsection{Uniform moment estimates} \label{sec:umom} In this subsection we provide growth estimates for the moments of the \ac{SHE} started from $f\in \hypz$.

\begin{proposition}\label{umomf} Fix $\mathfrak{q}\in ((0,\infty)\times (0,1])^{\mathbb{N}}$, $T>0$ and $k\ge 2$. There exists a constant $\Delta=\Delta(\mathfrak{q}_k,k,T)>0$ such that for all $s\in [0,T]$, for all functional $f\in \hypz$ and $y\in \R$ we have 
	\begin{align}\label{umomfm}
		\Ex\left[[\calZ^{f}_s(y)]^k\right] \le \Delta e^{\Delta |y|^{2-\delta_k}}.
	\end{align}
\end{proposition}

Note that the above proposition does not cover the narrow wedge initial data case. For that initial data we have the following moment estimate.
\begin{customprop}{\ref{umomf}-nw}
	\label{umom} Fix $T>0$ and $k\in \Z_{>0}$. Consider the \ac{SHE} $\calZ^{\nw}_t(x)$ started from delta initial data \eqref{deltaic}. There exists a constant $\Con=\Con(k,T)>0$ such that for all $s\in (0,T]$ and $y\in \R$ we have 
	\begin{align}\label{nwmom}
		\Ex\left[ (2\pi s)^{k/2}[\calZ^{\nw}_s(y)]^k\exp\left(\tfrac{ky^2}{2s}\right)\right]\le \Con<\infty.
	\end{align}
\end{customprop}
\begin{proof} By Proposition 1.4 in \cite{acq} we have $\calZ^{\nw}(s,y)e^{\frac{y^2}{2s}}\stackrel{d}{=} \calZ^{\nw}(s,0)$. {It thus suffices to prove
\begin{align}\label{ditto}
		\sup_{s\in (0,T]} (2\pi s)^{k/2}\Ex\left[ [\calZ^{\nw}_s(0)]^k\right]\le \Con<\infty.
	\end{align}
Precise moment estimates for $\calZ^{\nw}_s(0)$ are available in the literature; see Lemma 4.1 in \cite{kpzgen} for example. Indeed, Lemma 4.1 in \cite{kpzgen} shows that there exists $\Con=\Con(k,T)>0$ such that $\Ex\left[ [\calZ^{\nw}_s(0)]^k\right] \le \Con$ for $s\in [\pi,T]$ and $\Ex\left[ [\calZ^{\nw}_s(0)]^k\right] \le \Con s^{-k/2}$ for $s \in (0,\pi]$. These two estimates together verifies \eqref{ditto}.}
\end{proof}

In order to extend our moment estimate from narrow wedge initial data to functional initial data we rely on the convolution formula (Proposition \ref{p:conv}).

\begin{proof}[Proof of Proposition \ref{umomf}] {To get the moment bounds for general initial data, we use  Proposition \ref{p:conv} and a clever H\"older trick inspired from \cite{yier} (see Eq (2.2)).  Using Proposition \ref{p:conv} we have
\begin{equation}
    \label{f11}
    \begin{aligned}
			\Ex[\calZ_s^f(y)^k] & = \Ex\left[\left(\int_{\R} \calZ_s^{\nw,z}(y)e^{f(z)} \d z \right)^k \right] \\ & =\Ex\left[\left(\int_{\R} e^{-\frac{(y-z)^2}{4s}}e^{\frac{(y-z)^2}{4s}}\calZ_s^{\nw,z}(y)e^{f(z)} \d z \right)^k \right].
	\end{aligned}
\end{equation}		
  We now apply H\"older inequality to the terms $e^{-\frac{(y-z)^2}{4s}}$  and $e^{\frac{(y-z)^2}{4s}}\calZ_s^{\nw,z}(y)e^{f(z)}$ with H\"older exponent $k$ and $\ell=\frac{k}{k-1}$ to get
		\begin{align}
			\nonumber
			\Ex[\calZ_s^f(y)^k] & \le \left(\int_{\R} e^{-\frac{\ell(y-z)^2}{4s}} dz \right)^{\frac{k}\ell}\Ex\left[\int_{\R} e^{\frac{k(y-z)^2}{4s}}[\calZ_s^{\nw,z}(y)]^k e^{kf(z)} \d z \right] \\ \label{peneq} & = \left(\int_{\R} e^{-\frac{\ell(y-z)^2}{4s}} dz \right)^{k-1}\int_{\R} e^{\frac{k(y-z)^2}{4s}}M_k^f(z)\Ex[\calZ_s^{\nw,z}(y)^k] \d z \\  & = \left(\tfrac{4s\pi}{\ell} \right)^{\frac{k-1}2}\int_{\R} e^{\frac{k(y-z)^2}{4s}}M_k^f(z)\Ex[\calZ_s^{\nw,z}(y)^k] \d z \label{f1}
		\end{align}
		where $M_k^f(x)$ is defined in \eqref{mkf}.} {Note that the equality in \eqref{peneq} follows by interchanging the expectation and integral (as the integrand is nonnegative) and utilizing the fact that the initial data $f$ is independent of the underlying space-time white noise}. Using \eqref{idenp} and applying Proposition \ref{umom} we have
	\begin{align}
		\mbox{r.h.s.~of \eqref{f1}} & \le \Con\tfrac1{(2\pi s)^{1/2}}\int_{\R} M_k^f(z)\exp\left({-\tfrac{k(y-z)^2}{4s}}\right)\d z \label{3.4.1}
	\end{align}
	where $\Con>0$ depends on $k,T$.  Using assumption \eqref{mkcond} followed by a change of variable $u\mapsto \sqrt{k}(z-y)/\sqrt{2s}$ we have
	\begin{align*}
		\mbox{r.h.s.~of \eqref{3.4.1}} 
		& = \Con\tfrac1{(2\pi s)^{1/2}}\sqrt{\tfrac{2s}{k}}e^{\lambda_k}\int_{\R} \exp\left(\lambda_k|u\sqrt{\tfrac{2 s}{k}}+{y}|^{2-\delta_k}-\tfrac{u^2}{2}\right)\d u  \\ & \le \Con e^{\lambda_k} e^{\Delta |y|^{2-\delta_k}}\int_{\R} \exp\left(\Delta|u|^{2-\delta_k}-\tfrac{u^2}{2}\right)\d u,
	\end{align*}
	where in the last line we used the fact that one can choose $\Delta>0$ depending on $\lambda_k$ and $T$ such that $\lambda_k |u\sqrt{\frac{2 s}{k}}+y|^{2-\delta_k} \le \Delta (|u|^{2-\delta_k}+|y|^{2-\delta_k})$ for all $s\in [0,T]$ and for all $u,y\in \R$.
	Clearly the above integral is uniformly bounded by a constant depending on $\Delta$ and $\delta_k$. Thus adjusting the constant $\Delta$ we get the desired bound.
\end{proof}

\subsection{Spatial Regularity estimate}\label{sec:spmom} In this subsection we prove regularity estimates for the moments of $|\calZ_t^f(x)-\calZ_t^f(0)|$. As before we rely on the convolution formula to bypass narrow wedge results to functional data. For proving regularity estimates for narrow wedge data, we use tail estimates for differences of height functions in the scaled KPZ equation noted in \cite{kpztime}. Towards this end, for narrow wedge initial data, we introduce the $1:2:3$ scaled version of the KPZ equation:
\begin{align}\label{321}
	\h_t(x):=\frac{\log \calZ_{t}^{\nw}(t^{2/3}x)+\frac{t}{24}}{t^{1/3}}.
\end{align}
The above $1:2:3$ scaling of fluctuations, space, and time are the characteristics of the models in KPZ universality class.  Being an integrable model, $\h_t(x)$ itself enjoys several properties. It admits exact one-point distribution formula and moment formulas \cite{acq}. Furthermore, it can be viewed as the top curve of the KPZ line ensemble that enjoys a type of Brownian Gibbs property \cite{ch16}. Exploiting such a rich structure several fine estimates for $\h_t(x)$ are obtained  \cite{kpzgen}, \cite{kpztime}, and \cite{dg}. We utilize them to obtain moment estimates in our setting.

\begin{proposition}\label{spmomf} Fix $\mathfrak{q}\in ((0,\infty)\times (0,1])^{\mathbb{N}}$, $T>1$, $\theta\in (0,\frac12)$ and $k\ge 2$. There exists a constant  $\Con=\Con(\mathfrak{q}_k,k,T,\theta)>0$ such that for all functional $f\in \hypz$ we have
	\begin{align*}
		\sup_{t\in [T^{-1},T], 0<|x|\le \frac1\Con} \Ex\left[ \frac{|\calZ_t^{f}(x)-\calZ_t^{f}(0)|^k}{|x|^{k\theta}}\right]= \Con<\infty.
	\end{align*}
\end{proposition}

\begin{customprop}{\ref{spmomf}-nw}\label{spmom} Fix $T>1$ and $k\in \Z_{>0}$. There exists a constant $\Con=\Con(k,T)>0$  such that for all $y\in \R$ we have
	\begin{align*}
		\sup_{t\in [T^{-1},T], x\neq 0} \Ex\left[ \frac{|\calZ_t^{\nw}(x+y)-\calZ_t^{\nw}(y)|^k}{|x|^{k/2}}\right]= \Con<\infty.
	\end{align*}
\end{customprop}

We use tail estimates of $\h_t(x)$ and its spatial or temporal increments and convert them to moment estimates for $\calZ_t^{\nw}(y)$ via the following Lemma.

\begin{lemma}\label{lhz2} Let $X : \R\times\R \to \R$ be a  stochastic process. Fix $M>1$ and let $I$ be an interval of $\R$. Suppose for all $s,t\in I$ with $|t-s|< 1$ and for all $m>0$ we have
	\begin{align}\label{htail}
		\Pr\left(|X(s,t)|\ge m|s-t|^{\beta}\right) \le \Con\exp\left(-\tfrac1\Con m^{\alpha}\right)
	\end{align}
	for some constant $\Con>0$, $\beta>0$ and $\alpha>1$.	Then for all $k\in \Z_{>0}$ we have
	\begin{align}\label{htoz}
		\sup_{p\in [M^{-1},M]}\sup_{\substack{s\neq t \in I\\ |t-s|\le 1}} \Ex\left[\frac{\left|e^{pX(s,t)}-1\right|^k}{|s-t|^{k\beta}}\right] =R(\alpha,\beta,\Con,k,M) <\infty.
	\end{align}
\end{lemma}
We first complete the proof of Proposition \ref{spmom} assuming Lemma \ref{lhz2}. 

\begin{proof}[Proof of Proposition \ref{spmom}] Fix $T>1$, $t\in [T^{-1},T]$ and $x,y\in \R$ with $x\neq 0$. By Proposition \ref{umom}, Proposition \ref{spmom} follows for all $x$ bounded away from $0$. Thus we may assume $|x|\le T^{-2/3}$ and set $u= t^{-2/3}x$ and $v=t^{-2/3}y$. As $t\in [T^{-1},T]$, we have $|u|\le T^{2/3}|x|\le 1$. {Using the scaling in \eqref{321}, we have the following relation:
\begin{align*}
    \frac{\calZ_t^{\nw}(x+y)e^{\frac{(x+y)^2}{2t}}}{\calZ_t^{\nw}(y)e^{\frac{y^2}{2t}}}=\exp\bigg(t^{1/3}\big(\h_t(u+v)+\tfrac{(u+v)^2}{2}-\h_t(v)-\tfrac{v^2}{2}\big)\bigg).
\end{align*}
We introduce the parabola in the above equation as $\h_t(w)+\frac{w^2}{2}$ is stationary in $w$ and one has tail estimates for the two-point difference of the process $\h_t(w)+\frac{w^2}{2}$. Let us temporarily set
\begin{align*}
    A:=\calZ_t^{\nw}(x+y)e^{\frac{(x+y)^2-y^2}{2t}}=\calZ_t^{\nw}(y)\exp\bigg(t^{1/3}\big(\h_t(u+v)+\tfrac{(u+v)^2}{2}-\h_t(v)-\tfrac{v^2}{2}\big)\bigg).
\end{align*}
 In view of the above equality, we have the following decomposition:
	\begin{equation}\label{e:rrw0}
		\begin{aligned}
			\frac{\calZ_t^{\nw}(x+y)-\calZ_t^{\nw}(y)}{\sqrt{|x|}} & = \frac{\calZ_t^{\nw}(x+y)-A}{\sqrt{|x|}}+\frac{A-\calZ_t^{\nw}(y)}{\sqrt{|x|}} \\ &  = \frac{\calZ_t^{\nw}(x+y)\big[1-e^{\frac{(x+y)^2-y^2}{2t}}\big]}{\sqrt{x}} \\ & \hspace{2cm}+ \calZ_t^{\nw}(y)\frac{\left[e^{t^{1/3}[\h_t(u+v)+\frac{(u+v)^2}{2}-\h_t(v)-\frac{v^2}{2}]}-1\right]}{\sqrt{|x|}}.
		\end{aligned}
	\end{equation}
As $|a+b|^k \le 2^{k-1}(|a|^k+|b|^k)$, it suffices to provide bounds for the $k$-th moment for each of the two terms in r.h.s.~of \eqref{e:rrw0} separately. As the second term is a product of two terms, using $|cd|^k \le |c|^{2k}+|d|^{2k}$, it is enough to provide moment bounds for $\mathcal{Z}_t^{\nw}(y)$ and $\left[e^{t^{1/3}[\h_t(u+v)+\frac{(u+v)^2}{2}-\h_t(v)-\frac{v^2}{2}]}-1\right]$ separately.	Now Proposition \ref{umom} ensures all moments of $\mathcal{Z}_t^{\nw}(y)$ are uniformly bounded. On the other hand, By Theorem 1.3 in \cite{kpztime} for all $v\in \R$ and $u\in [-1,1]$ we have
	\begin{align*}
		\Pr(|\h_t(u+v)+\tfrac{(u+v)^2}{2}-\h_t(v)-\tfrac{v^2}{2}|\ge m\sqrt{|u|}) \le \Con\exp\left(-\tfrac1\Con m^{3/2}\right),
	\end{align*}
where the constant $\Con>0$ depends only on $T$. Since $u=t^{-2/3}x$ and $t\in [T^{-1},T]$. Using the above estimate, in view of Lemma \ref{lhz2}, we get that 
	\begin{align*}
		\sup_{\substack{t\in [T^{-1},T]\\ 0<|x|\le T^{-2/3}, v\in \R}}\Ex\left[\frac{|e^{t^{1/3}[\h_t(t^{-2/3}x+v)+\frac{(t^{-2/3}x+v)^2}{2}-\h_t(v)-\frac{v^2}{2}]}-1|^{2k}}{|x|^{k}}\right] <\infty.
	\end{align*}
This leads to a uniform moment bound for the second term in \eqref{e:rrw0}. For the first term, using Proposition \ref{umom} we have
	\begin{align*}
		\big|1-e^{\frac{(x+y)^2-y^2}{2t}}\big|^k|x|^{-k/2}\Ex[\calZ_t^{\nw}(x+y)^k] & \le \Con|x|^{-k/2}\big|1-e^{\frac{(x+y)^2-y^2}{2t}}\big|^ke^{-\frac{k(x+y)^2}{2t}} \\ & = \Con\bigg[|x|^{-1/2}\big|e^{-\frac{(x+y)^2}{2t}}-e^{-\frac{y^2}{2t}}\big|\bigg]^k\\ & \le \Con\bigg[|x|^{1/2}\cdot\sup_{y\in \R, t\in [T^{-1},T]}\big|\partial_y (e^{-\frac{y^2}{2t}})\big|\bigg]^k,
	\end{align*}
	where the last inequality follows by applying the mean value theorem and then taking supremum over the derivative. The last term is uniformly bounded over $0< x\le T^{-2/3}\le 1$ (as $T\ge 1$). This completes the proof.}
\end{proof}

\begin{proof}[Proof of Lemma \ref{lhz2}] {Due to \eqref{htail}, for all $p\in [M^{-1},M]$, for all $m>0$, and $s\neq t\in I$ 
we have
$$\Pr\left(|pX(s,t)|\le m|s-t|^{\beta}\right) \le \Con\exp\left(-\tfrac1\Con m^{3/2}p^{-3/2}\right) \le \til{\Con}\exp\left(-\tfrac1{\til{\Con}} m^{3/2}\right)$$ where $\til{\Con}=M^{3/2}\Con$. Thus it is enough to show \eqref{htoz} assuming $p=1$. Toward this end, we will show 
	$$\frac{e^{X(s,t)}-1}{|s-t|^{\beta}}$$
	has good enough upper and lower tail decay uniform in $s,t\in I$ with $0<|t-s|<1$. Assume $m\ge 2$. For the upper tail using \eqref{htail} we have
	\begin{align} \nonumber
		\Pr\left(e^{X(s,t)}-1\ge m{|s-t|^{\beta}}\right) & \le \Pr\left(|X(s,t)|\ge \log(1+m|s-t|^{\beta})\right) \\ & \le  \Con \exp\left(-\tfrac1\Con |t-s|^{-\alpha\beta}{\log^{\alpha}(1+m|t-s|^{\beta})}\right). \label{e:a11} 
	\end{align}	
		Note that for $u\ge 1$ and $|t-s|\le 1$, we have 
	$\frac{|t-s|^{\beta}}{1+u|t-s|^{\beta}} \ge \frac{|t-s|^{\beta}}{2u}$. Integrating both sides w.r.t.~$u$ from $1$ to $m$ leads to
 $$\log(1+m|t-s|^{\beta}) \ge \tfrac12|t-s|^{\beta}\log m+\log(1+|t-s|^{\beta}) \ge \tfrac12|t-s|^{\beta}\log m.$$ Thus, $\mbox{r.h.s.~of \eqref{e:a11}}  \le \Con \exp\left(-\frac1{2^{\alpha}\Con}\log^{\alpha} m\right)$.}

 \medskip

	{For the lower tail observe that if $m|s-t|^{\beta}\ge 1$, then $e^{X(s,t)}-1 > -m|s-t|^{\beta}$. Thus we may assume $m|s-t|^{\beta}<1$. Take $u\ge 1$ with $u|s-t|^{\beta}<1$.  We have $\frac{{|s-t|^{\beta}}}{1-u{|s-t|^{\beta}}} \ge \frac{{|s-t|^{\beta}}}u$. Integrating both sides w.r.t.~$u$ from $1$ to $m$ leads to $$-\log(1-m\sqrt{|s-t|^{\beta}}) \ge \sqrt{|s-t|^{\beta}}\log m-\log(1-|s-t|^{\beta}) \ge |s-t|^{\beta}\log m.$$ 
	Using this inequality and \eqref{htail} we get
	\begin{align*} \nonumber
		\Pr\left(e^{X(s,t)}-1\le -m{|s-t|^{\beta}}\right) & \le \Pr\left(|X(s,t)|\ge -\log(1-m|s-t|^{\beta})\right) \\ & \le  \Con \exp\left(-\tfrac1\Con |t-s|^{-\alpha\beta}{[-\log(1-m|t-s|^{\beta})]^{\alpha}}\right) \\ & \le \exp\left(-\tfrac1{\Con}\log^{\alpha} m\right). 
	\end{align*}	
As $\alpha>1$, both the upper and lower tail decay faster than any polynomial. This proves \eqref{htoz}.} 
\end{proof}

\begin{proof}[Proof of Proposition \ref{spmomf}] Fix any $f\in \hypz$.  Fix $|x|\le \frac12$. {By convolution formula (Proposition \ref{p:conv}) and the elementary inequality $|\int g(y)\d y|\le \int |g(y)|\d y$ we have
		\begin{align*}
			|\calZ_t^{f}(x)-\calZ_t^{f}(0)| & = \left|\int_{\R} e^{f(y)}[\calZ_t^{\nw, y}(x)-\calZ_t^{\nw,y}(0)] \d y\right| \\ & \le  \int_{\R} e^{f(y)}\left|\calZ_t^{\nw, y}(x)-\calZ_t^{\nw,y}(0)\right| \d y. 
		\end{align*}
		Similar to \eqref{f11} and \eqref{f1}, we first rewrite the above integrand $e^{f(y)}\left|\calZ_t^{\nw, y}(x)-\calZ_t^{\nw,y}(0)\right|$ as product of $e^{\frac{y^2}{4t}}e^{f(y)}\left|\calZ_t^{\nw, y}(x)-\calZ_t^{\nw,y}(0)\right|$ and $e^{-\frac{y^2}{4t}}$. Then apply H\"older inequality with exponents $k$ and $\ell=\frac{k}{k-1}$ to get
		\begin{align}
			\Ex\left[\left|\calZ_t^{f}(x)-\calZ_t^{f}(0)\right|^k\right] \le \left(\tfrac{4s\pi}{\ell}\right)^{\frac{k-1}2}\int_{\R} e^{\frac{ky^2}{4t}}M_k^f(y) \Ex\big[|\calZ_t^{\nw, y}(x)-\calZ_t^{\nw,y}(0)|^k\big]\,\d y, \label{350}
		\end{align}
		where $M_k^f(x)$ is defined in \eqref{mkf}.}	Using 	\eqref{idenp} and \eqref{mkcond} we have
	\begin{align}\label{3.5.1}
		\mbox{r.h.s.~of \eqref{350}} & \le \Con \int_{\R} \exp\left(\tfrac{k y^2}{4t}+\lambda_k(1+|y|^{2-\delta_k})\right)\Ex\big[|\calZ_t^{\nw}(x+y)-\calZ_t^{\nw}(y)|^k\big]\d y.
	\end{align}
{By spatial regularity estimate from Proposition \ref{spmom} we have	
 \begin{align}\label{awe}
		\Ex\big[|\calZ_t^{\nw}(x+y)-\calZ_t^{\nw}(y)|^k\big] \le \Con\cdot |x|^{k/2}.
	\end{align}
On the other hand, by moment estimate from Proposition \ref{umom} we have	
 \begin{align*}
		\Ex\big[|\calZ_t^{\nw}(x+y)-\calZ_t^{\nw}(y)|^k\big] & \le 2^{k-1}\left[\Ex\big[|\calZ_t^{\nw}(x+y)|^k\big]+\Ex\big[|\calZ_t^{\nw}(y)|^k\big]\right] \\ & \le \Con\left[\exp(-\tfrac{k(x+y)^2}{2t})+\exp(-\tfrac{k y^2}{2t})\right]
	\end{align*}
where the constant $\Con>0$ depends only on $k,T$. Here we crucially used the fact that $t$ has a lower bound, i.e., $t\ge T^{-1}$ to derive the constant $\Con$ free of $t$. Now as $|x|\le \frac12$, we may choose $y_0(\mathfrak{q}_k,k,T)>0$, such that for all $|y|\ge y_0$ and for all $t\in [T^{-1},T]$ we have $$\exp(-\tfrac{k(x+y)^2}{2t})+\exp(-\tfrac{k y^2}{2t}) \le \Con\exp(-\lambda_k(1+|y|^{2-\delta_k})-\tfrac{k y^2}{4t}-\tfrac{k y^2}{8t})$$
for some $\Con$ depending on $\mathfrak{q}_k,k,T$.} Thus combining this estimate with \eqref{awe}, we have that for all $|y|\ge y_0$ and $|x|\le \frac12$,
	\begin{align*}
		\Ex\big[|\calZ_t^{\nw}(x+y)-\calZ_t^{\nw}(y)|^k\big] \le \Con\cdot \min \left\{|x|^{k/2}, \exp\left(-\lambda_k(1+|y|^{2-\delta_k})-\tfrac{k y^2}{4t}-\tfrac{k y^2}{8t}\right)\right\}.
	\end{align*}
	Thus, 
	\begin{align*}
		\mbox{r.h.s.~of \eqref{3.5.1}} \le \Con\int_{-y_0}^{y_0} |x|^{k/2} e^{\lambda_k|y|^{2-\delta_k}}\,\d y+\Con\int_{|y|\ge y_0} \min\left\{|x|^{k/2}e^{\lambda_k|y|^{2-\delta_k}},e^{-\tfrac{k y^2}{8T}}\right\}\,\d y.
	\end{align*}
	The first term is indeed bounded by $\Con \cdot |x|^{k/2}$. For the second term, we set $y_1=\sqrt{2T\log (1/x)}$ and observe  
	\begin{align*}
		\int_{|y|\ge y_1} \min\left\{|x|^{k/2}e^{\lambda_k|y|^{2-\delta_k}},e^{-\tfrac{k y^2}{8T}}\right\}\,\d y \le \int_{|y|\ge y_1} e^{-\tfrac{k y^2}{8T}}\,\d y \le \Con|x|^{k},
	\end{align*}
	\begin{align*}
		\int_{|y|\le y_1} \min\left\{|x|^{k/2}e^{\lambda_k|y|^{2-\delta_k}},e^{-\tfrac{k y^2}{8T}}\right\}\,\d y \le 2y_1|x|^{k/2} e^{\lambda_k|y_1|^{2-\delta_k}} \le \Con |x|^{k\theta}.
	\end{align*}
	The last inequality is true for all small enough $x$ depending on $\lambda_k$, $\delta_k$, $\theta$, and $T$. This is due to the fact $y_1$ and $e^{\lambda_k|y_1|^{2-\delta_k}}$ are both sub-polynomial order. This completes the proof. 
\end{proof}

\subsection{Temporal regularity estimate}\label{sec:tpmom} The main goal of this section is to prove the following proposition that provides temporal regularity estimates for the \ac{SHE} started from initial data $f$ in $\hypz$ class defined in Definition \ref{hyp}.

\begin{proposition}\label{tpmomf} Fix $\mathfrak{q}\in ((0,\infty)\times (0,1])^{\mathbb{N}}$, $T>1$ and $k\ge 2$.  There exists a constant  $\Con=\Con(\mathfrak{q}_k,k,T)>0$ such that for all $f\in \hypz$ we have
	\begin{align*}
		\sup_{s\neq t\in [T^{-1},T]} \Ex\left[ \frac{|\calZ_s^{f}(0)-\calZ_t^{f}(0)|^k}{|s-t|^{k/4}}\right]= \Con<\infty.
	\end{align*}
\end{proposition}

The proof of Proposition \ref{tpmomf} relies on some heat kernel estimates which we record below.

\begin{lemma}\label{heat} Fix $\theta\in (0,\frac12)$ and $\Delta,\delta_k>0$, one can find constants $\rho_0=\rho_0(\theta,\Delta,\delta_k)>0$ and $\Con=\Con(\theta,\Delta,\delta_k)>0$ such that for all $\rho\in (0,\rho_0)$ we have
	\begin{align}\label{thet}
		\qquad \int_0^{\rho}\int_{\R}p^2_{r}(y)\min\{ |y|^{2\theta}, e^{\Delta |y|^{2-\delta_k}}\}\,\d y\,\d r \le \Con \rho^{\frac12+\theta}.
	\end{align}
	Additionally we also have	\begin{align} \label{nthet}
		\int_0^{\rho}\int_{\R}p^2_{r}(y)e^{\Delta |y|^{2-\delta_k}}\,\d y\,\d r \le \Con \rho^{\frac12},
	\end{align}
	for all $\rho\in (0,\rho_0)$ where now $\rho_0$ and $\Con$ can be chosen free of $\theta$.
\end{lemma}
\begin{lemma}\label{heat2} Fix $\mathfrak{q}\in ((0,\infty)\times (0,1])^{\mathbb{N}}$, $T>1$, and $a\in [\frac12,\infty]$. There exist constants $\e_0(\mathfrak{q}_k,k,T)>0$ and $\Con(\mathfrak{q}_k,k,T)>0$, such that for all $\e\le \e_0$ and for all $f\in \hypz$ we have
	\begin{align*}
		\int_0^{t-\e^a}\int_{\R}[p_{t+\e-s}(x)-p_{t-s}(x)]^2\big[1+\Ex[\calZ_s^f(x)^k]^{\frac2k}\big]\,\d x\,\d s \le \Con \cdot \e^{2-\frac{3}2 \min(a,1)}.
	\end{align*}
\end{lemma}

Besides the above lemmas, proof of Proposition \ref{tpmomf} also utilizes Burkholder-Davis-Gundy (BDG) inequality \cite[Lemma 2.3]{con}.  Since we will make repeated use of this inequality throughout, we recall it here for readers' convenience. 

{\begin{theorem}[BDG inequality] \label{thm9} For every $t \ge 0$, let $\mathcal{F}_t^0$ denote the sigma-algebra generated by every Wiener
integral of the form 
$$\int_{(0,t)\times \R} \psi_s(y) \xi(\d s, \d y)$$ as $\psi$ ranges over all elements of $L^2(\R_{+}\times \R)$. We complete every such sigma-algebra, and make the filtration
$\{\mathcal{F}_t\}_{t\ge0}$ right continuous in order to obtain the ``Brownian filtration'' $\mathcal{F}$ that corresponds to the white noise $\xi$.
Let $\Psi := \{\Psi_t(x)\}_{t\ge 0,x\in \R}$ be a predictable random field with respect to $\mathcal{F}$. Then, for every real number $k \in [2 , \infty)$, we have the following inequality:
\begin{align*}
    \Ex\left[\left|\int_{(0,t)\times \R} \Psi_s(y)\xi(\d s,\d y)\right|^k\right]^{\frac2k} \le 4k\int_{(0,t)\times \R} \Ex\left[\left|\Psi_s(y)\right|^k\right]^{\frac2k}\,\d s\,\d y.
\end{align*}
\end{theorem}}

Let us now first complete the proof of Proposition \ref{tpmomf} assuming Lemmas \ref{heat} and \ref{heat2}.

\begin{proof}[Proof of Proposition \ref{tpmomf}] By Proposition \ref{umomf} and \ref{umom}, Proposition \ref{tpmomf} holds for $s,t$ with $|s-t|$ bounded away from zero. So we assume $s>t$ and $s-t$ is small enough. Set $s=t+\e$ and from \eqref{fkdc} observe that
{\begin{equation}
    \label{ztef}
    \begin{aligned}
		\calZ_{t+\e}^f(0)-\calZ_t^f(0) & = \int_{\R} [p_{t+\e}(y)-p_t(y)]e^{f(y)}\,\d y+ \int_t^{t+\e}\int_{\R} p_{t+\e-s}(y)\calZ_s^f(y)\xi(\d s,\d y) \\ & \hspace{1cm}+\int_0^{t}\int_{\R}[p_{t+\e-s}(y)-p_{t-s}(y)]\calZ_s^f(y)\xi(\d s,\d y).
	\end{aligned}
\end{equation}
Our goal is to control the moments for each term appearing  in r.h.s.~of \eqref{ztef}.
For the first term, we claim that
\begin{align} \label{claimg}
    \Ex\bigg[\big|\int_{\R} \left[p_{t+\e}(y)-p_t(y)\right]e^{f(y)}\d y\big|^k\bigg] \le \Con \cdot \e^k
\end{align}
where $\Con$ depends on $k,T, \lambda_k$ and $\delta_k$. Towards this end, we first note that
		\begin{align*}
			\left|\int_{\R} [p_{t+\e}(y)-p_t(y)] e^{f(y)} \d y \right| \le \int_{\R} |p_{t+\e}(y)-p_t(y)| e^{f(y)} \d y. 
		\end{align*}
		Writing  $|p_{t+\e}(y)-p_t(y)| e^{f(y)}$ as product of $e^{\frac{ y^2}{8T}}|p_{t+\e}(y)-p_t(y)| e^{f(y)}$ and $e^{-\frac{y^2}{8T}}$, and applying H\"older inequality with exponent $k$ and $\ell=\frac{k}{k-1}$ we get
		\begin{align*}
			\left(\int_{\R} |p_{t+\e}(y)-p_t(y)| e^{f(y)} \d y \right)^k \le \left(\int_{\R} e^{-\frac{\ell y^2}{8T}}\right)^{k-1}\int_{\R} e^{\frac{k y^2}{8T}}|p_{t+\e}(y)-p_t(y)|^k e^{kf(y)} \d y. 
		\end{align*}
		Hence  we have
		\begin{align}
			\label{yier}
			\Ex\bigg[\big|\int_{\R} \left[p_{t+\e}(y)-p_t(y)\right]e^{f(y)}\d y\big|^k\bigg] \le \Con \int_{\R} e^{\frac{k y^2}{8T}}|p_{t+\e}(y)-p_t(y)|^k \Ex[e^{kf(y)}] \d y ,
	\end{align}
	where $\Con$ depends on $k,T$. But by mean value theorem we get that
	\begin{align*}
	    e^{\frac{k y^2}{8T}}|p_{t+\e}(y)-p_t(y)|^k\Ex[e^{kf(y)}] & \le \e^k \cdot \sup_{t\in [T^{-1},T+1]} e^{\frac{k y^2}{8T}}M_k^f(y)|\partial_t p_t(y)|^k \\ & \le \Con \e^k \cdot  e^{\frac{k y^2}{8T}}e^{\lambda_k(1+|y|^{2-\delta_k})}y^2e^{-\frac{ky^2}{2T}},
	\end{align*}
	where we have used the condition \eqref{mkcond} in the last inequality above. Integrating both sides of the above equation over $y\in \R$, we get that  r.h.s.~of \eqref{yier} is at most $\Con \cdot \e^k$ where $\Con$ depends on $k,T, \lambda_k$ and $\delta_k$. This verifies \eqref{claimg}.}

	To compute the moments of the second and third term in r.h.s.~of \eqref{ztef}, we apply BDG inequality (Theorem \ref{thm9}) on each of the above terms to get
 \begin{align*}
     & \Ex\left[\left|\int_t^{t+\e}\int_{\R} p_{t+\e-s}(y)\calZ_s^f(y)\xi(\d s,\d y)\right|^k\right]^{\frac2k} \\ & \hspace{3cm}\le \Con\int_t^{t+\e}\int_{\R}p_{t+\e-s}(y)^2\Ex[\calZ_s^f(y)^k]^{\frac2k} \d y\,\d s, \\
     & \Ex\left[\left|\int_0^{t}\int_{\R}[p_{t+\e-s}(y)-p_{t-s}(y)]\calZ_s^f(y)\xi(\d s,\d y)\right|^k\right]^{\frac2k} \\ & \hspace{3cm}\le \Con\int_0^{t}\int_{\R}[p_{t+\e-s}(y)-p_{t-s}(y)]^2\Ex[\calZ_s^f(y)^k]^{\frac2k}\d y\,\d s. 
 \end{align*}
 In view of the above moment estimates along with \eqref{claimg}, leads to moment estimates for the temporal difference in view of the identity in \eqref{ztef}. Indeed, using the elementary inequalities $|a+b+c|^k \le 3^{k-1}(|a|^k+|b|^k+|c|^k)$ and $|x+y+z|^{\frac2k} \le |x|^{\frac2k}+|y|^{\frac2k}+|z|^{\frac2k}$ for $k\ge 2$ we get that
	\begin{align*}
		\Ex\left[\big|\calZ_{t+\e}^f(0)-\calZ_t^f(0)\big|^k\right]^{\frac2k} & \le \Con \cdot \e^k+\Con\int_t^{t+\e}\int_{\R}p_{t+\e-s}(y)^2\Ex[\calZ_s^f(y)^k]^{\frac2k} \d y\,\d s \\ & \hspace{2cm}+\Con\int_0^{t}\int_{\R}[p_{t+\e-s}(y)-p_{t-s}(y)]^2\Ex[\calZ_s^f(y)^k]^{\frac2k}\d y\,\d s.
	\end{align*}
	By the moment estimates from Proposition \ref{umom} and Proposition \ref{umomf} we see that the first double integral above is at most $\int_0^{\e}\int_\R p_r^2(y)e^{\Delta |y|^{2-\delta}}\,\d y\, \d r$ which is at most $\Con \e^{\frac12}$ by Lemma \ref{heat} with $\rho=\e$. The second double integral above is at most $\Con \e^{\frac12}$ as well from Lemma \ref{heat2} with $a=\infty$. This completes the proof.	
\end{proof}

\begin{proof}[Proof of Lemma \ref{heat}] Suppose $c_0$ be such that $|y|^{2\theta}\le e^{\Delta|y|^{2-\delta_k}}$ for all $|y|\le c_0$. We separate the $y$ integral into two parts depending on $|y|\le c_0$ and $|y|\ge c_0$. Making a change of variable $x\mapsto \frac{y}{\sqrt{r}}$ we see that
	\begin{align*}
		\int_0^{\rho}\int_{|y|\le c_0}p^2_{r}(y) |y|^{2\theta}\,\d y\,\d r = \frac1{2\pi}\int_0^{\rho} r^{\theta-\frac12}\int_{|x|\le c_0/\sqrt{r}}e^{-x^2} |x|^{2\theta}\,\d x\,\d r \le \Con \int_0^{\rho} r^{\theta-\frac12}\,\d r \le \Con \rho^{\frac12+\theta}.
	\end{align*}
	For the $|y|\ge c_0$ integral, we can choose $\rho_0$ depending on $c_0$ and other parameters so that the following two conditions hold:
	\begin{itemize}
		\item $e^{-\frac{y^2}{r}} \le r^2e^{-\frac{y^2}{2r}}$ for all $r\le \rho \le \rho_0$ and $|y| \ge c_0$, and 
		\item $\frac{y^2}{4r} \ge \Delta |y|^{2-\delta_k}$ for all $|y|\ge c_0$ and $r\le \rho\le \rho_0$.
	\end{itemize}  
	Then for all $\rho\le \rho_0$ we have 
	\begin{align*}
		\int_0^{\rho}\int_{|y|\ge c_0}p^2_{r}(y)e^{\Delta |y|^{2-\delta_k}}\,\d y\,\d r & =\tfrac1{2\pi}\int_0^{\rho}\tfrac1{r}\int_{|y|\ge c_0}e^{-\frac{y^2}{r}}e^{\Delta |y|^{2-\delta_k}}\,\d y\,\d r \\ & \le \tfrac1{2\pi}\int_0^{\rho} r\int_{|y|\ge c_0}e^{-\frac{y^2}{4r}}\,\d y\,\d r \le \Con \int_0^{\rho} r^{\frac32}\,\d r \le \Con \rho^{\frac52}.
	\end{align*}
	This proves \eqref{thet}. For \eqref{nthet} we proceed in the exact same manner. Here we choose $c_0=1$. Then $|y|\le c_0$ integral is now bounded above by $e^{\Delta}\int_0^{\rho}\int_{|y|\le 1}p^2_{r}(y)\,\d y\,\d r \le \Con\rho^{\frac12}$. The other integral can be bounded exactly as before. This completes the proof.
\end{proof}

\begin{proof}[Proof of Lemma \ref{heat2}] Define
	\begin{align*}
		\Upsilon(B):=\int_B \int_{\R}[p_{t+\e-s}(y)-p_{t-s}(y)]^2\big[1+\Ex[\calZ_s^f(y)^k]^{\frac2k}\big]\,\d y\,\d s.
	\end{align*}
	
	We prove Lemma \ref{heat2} in two steps. In Step 1 we show that $\Upsilon((0,t-\e^{\frac14}))  \le \Con \e^{5/4}.$ Then in Step 2, we argue
	$\Upsilon([t-\e^{1/4},t-\e^a)) \le \Con \e^{2-\frac{3}2\min(a,1)}.$ Since $\Upsilon((0,t-\e^a))= \Upsilon((0,t-\e^{1/4}))+\Upsilon([t-\e^{1/4},t-\e^a))$, we get the desired bound.

	\medskip
	
	\noindent\textbf{Step 1}.
	For narrow wedge initial data applying Proposition \ref{umom} we get 
	\begin{align}
		\Upsilon((0,t-\e^{\frac14})) &  \le \Con\int_0^{t-\e^{\frac14}}\int_{\R}[p_{t+\e-s}(y)-p_{t-s}(y)]^2(1+s^{-1}e^{-y^2/s})\d y\, \d s, \label{bdg4}	
	\end{align}
	whereas for functional $f\in \hypz$ applying  Proposition \ref{umomf} we get
	\begin{align}
		\Upsilon((0,t-\e^{\frac14})) &  \le \Con\int_0^{t-\e^{\frac14}}\int_{\R}[p_{t+\e-s}(y)-p_{t-s}(y)]^2 e^{2\Delta |y|^{2-\delta_k}}\d y\, \d s, \label{bdg5}	
	\end{align}
 for some $\Delta(\mathfrak{q}_k,k,T)>0$.
	Note that $t-s\ge \e^{\frac14}$. Set
	\begin{align*}
		\chi:=\sup_{y\in \R,r\in [\e^{\frac14},T]} |p_{r+\e}(y)-p_{r}(y)| e^{\Delta|y|^{2-\delta_k}+|y|}.
	\end{align*}
	We claim that $\chi \le \Con \cdot \e^{\frac{5}8}$. As $t\in [T^{-1},T]$, assuming the claim we see that
	\begin{align*}
		\mbox{r.h.s.~of \eqref{bdg4}} \le \Con \cdot \chi^2 \int_0^T \int_{\R} [e^{-2|y|}+s^{-1}e^{-y^2/s-2|y|}]\,\d y\,\d s \le \Con \cdot \e^{\frac54},
	\end{align*}		
	and
	\begin{align*}
		\mbox{r.h.s.~of \eqref{bdg5}} \le \Con \cdot \chi^2 \int_0^T \,\d s\int_\R e^{-2|y|}\,\d y \le \Con \cdot \e^{\frac54},
	\end{align*}		
	which proves that 	$\Upsilon((0,t-\e^{\frac14})) \le \Con \e^{\frac54}$. Hence it suffices to prove the claim. Towards this end, we rely on mean value theorem. Observe that	
	\begin{align*}
		e^{\Delta |y|^{2-\delta_k}+|y|}\partial_{r}p_r(y) =\tfrac1{\sqrt{8\pi}} r^{-3/2}\left[\tfrac1ry^2e^{-\frac{y^2}{4r}}-e^{-\frac{y^2}{4r}}\right]e^{-\frac{y^2}{4r}+\Delta |y|^{2-\delta_k}+|y|}.
	\end{align*}
	Clearly $|\tfrac1ry^2e^{-\frac{y^2}{4r}}-e^{-\frac{y^2}{4r}}|$ is uniformly bounded as $y$ varies in $\R$ and $r$ varies in $(0,\infty)$. Furthermore 
	$$\sup_{y\in \R,r\in (0,T+1]} \exp\left(-\tfrac{y^2}{4r}+\Delta |y|^{2-\delta_k}+|y|\right) \le \sup_{y\in \R} \ \exp\left(-\tfrac{y^2}{4T+4}+\Delta |y|^{2-\delta_k}+|y|\right) \le \Con <\infty.$$
	Thus for all $r\in (0,T+1]$ and $y\in \R$ we have $|e^{\Delta |y|^{2-\delta_k}+|y|}\partial_{r}p_r(y)| \le \Con r^{-3/2}$. Hence
	$$\chi=\sup_{y\in \R,r\in [\e^{\frac14},T]}|p_{r+\e}(y)-p_{r}(y)| e^{\Delta|y|^{2-\delta_k}+|y|} \le \e \cdot \sup_{y\in \R,r\in [\e^{\frac14},T+1]} e^{\Delta |y|^{2-\delta_k}+|y|}|\partial_{r}p_r(y)|\le \Con \cdot \e^{\frac58}.$$ This completes our work for this step.  
	
	\medskip
	
	\noindent\textbf{Step 2.} In this step we prove that $\Upsilon([t-\e^{1/4},t-\e^a)) \le \Con \e^{2-\frac{3}2\min(a,1)}.$ We may choose $\e$ small enough so that $s\ge t-\e^{\frac14} \ge \frac1{2T}$. Hence the range of integration is bounded away from zero. We may thus combine the estimates from Proposition \ref{umom} and Proposition \ref{umomf} to conclude that for all $f\in \hypz$ we have
	\begin{align*}
		\Upsilon([t-\e^{1/4},t-\e^a))  \le \Con\int_{\e^{a}}^{\e^{\frac14}}\int_{\R}[p_{r+\e}(y)-p_{r}(y)]^2e^{2\Delta |y|^{2-\delta_k}}\d y\, \d r,
	\end{align*}
	for some $C, \Delta>0$ depending on $\mathfrak{q},k,T$ ({Note that this $\Delta$ might be different from the one appearing in \eqref{umomfm} as we have combined it with the estimate from \eqref{nwmom}}). We divide the $y$-integral into two parts base on $|y|\le 1$ and $|y|\ge 1$,
	\begin{align*}
		\Con\int_{\e^{a}}^{\e^{1/4}}\int_{|y|\le 1}[p_{r+\e}(y)-p_{r}(y)]^2e^{2\Delta |y|^{2-\delta_k}}\d y\, \d r  & \le \Con e^{2\Delta}\int_{\e^{a}}^{\e^{1/4}}\int_{ \R}[p_{r+\e}(y)-p_{r}(y)]^2\d y\, \d r \\ & = \tfrac{\Con e^{2\Delta}}{2\sqrt\pi}\int_{\e^{a}}^{\e^{1/4}}\left[\tfrac1{\sqrt{r+\e}}+\tfrac1{\sqrt{r}}-\tfrac2{\sqrt{r+\frac{\e}2}}\right]\, \d r  \\ & = \tfrac{\Con e^{2\Delta}}{\sqrt\pi}[f(\e^{1/4})-f(\e^{a})],
	\end{align*}
	where 
	\begin{align*}
		f(r) & :=\sqrt{r+\e}+\sqrt{r}-2\sqrt{r+\tfrac{\e}{2}} \\ & = (\sqrt{r+\e}-\sqrt{r+\tfrac{\e}2})+(\sqrt{r}-\sqrt{r+\tfrac{\e}2}) \\ & = \frac{\e/2}{\sqrt{r+\e}+\sqrt{r+\tfrac{\e}2}}-\frac{\e/2}{\sqrt{r}+\sqrt{r+\tfrac{\e}2}} \\ & = (\e/2)\frac{\sqrt{r}-\sqrt{r+\e}}{(\sqrt{r+\e}+\sqrt{r+\tfrac{\e}2})(\sqrt{r}+\sqrt{r+\tfrac{\e}2})} \\ & = -\tfrac12\e^2\left[(\sqrt{r}+\sqrt{r+\e})(\sqrt{r}+\sqrt{r+\tfrac\e2})(\sqrt{r+\e}+\sqrt{r+\tfrac\e2})\right]^{-1}. 
	\end{align*}
	From the above expression we see that $f(r)$ is negative and decreasing. Furthermore, $|f(r)| \le \sqrt2\e^2(r+\e)^{-3/2} \le \sqrt2\e^2 \min(r,\e)^{3/2}$. Hence $$|f(\e^{\frac14})-f(\e^a)| \le |f(\e^a)| \le \sqrt2\e^2 \min(\e^a,\e)^{3/2} \le \Con\e^{2-\frac{3}2\min(a,1)}.$$ 
	We now turn towards the other integral: $|y|\ge 1$ region. Since $|y|\ge 1$, we may choose $\e_0=\e_0(k,\Delta,\delta_k)>0$ so that the following two conditions hold for all $\e\le \e_0$:
	\begin{itemize}
		\item $e^{-\tfrac{y^2}{r+\e}} \le \e^2 e^{-\tfrac{y^2}{2r+2\e}}$ for all $r\in (0,\e^{1/4}]$, and $|y|\ge 1$.
		\item $\frac{y^2}{4r+4\e} \ge2\Delta|y|^{2-\delta_k}$ for all $r\in (0,\e^{\frac14}]$, and $|y|\ge 1$.
	\end{itemize}
	 Assuming $\e \le \e_0$, the above condition forces 
 {$$p_{r+\e}(y)^2 e^{2\Delta |y|^{2-\delta_k}} \le \Con\e^2 (r+\e)^{-1/2}p_{r+\e}(y/2), \qquad  p_{r}(y)^2 e^{2\Delta |y|^{2-\delta_k}} \le \Con\e^2 r^{-1/2}p_{r}(y/2).$$
 Using $[p_{r+\e}(y)-p_r(y)]^2 \le p_{r+\e}(y)^2+p_r(y)^2$ we thus get that 
 $$[p_{r+\e}(y)-p_{r}(y)]^2e^{2\Delta |y|^{2-\delta_k}} \le \Con \e^2 [(r+\e)^{-1/2}p_{r+\e}(y/2)+r^{-1/2}p_{r}(y/2)].$$ Integrating both sides of the above equation over $|y|\ge 1$ and $r\in [\e^a,\e^{1/4}]$ we get
	\begin{align*}
		& \int_{\e^{a}}^{\e^{1/4}}\int_{|y|\ge 1}[p_{r+\e}(y)-p_{r}(y)]^2e^{2\Delta |y|^{2-\delta_k}}\d y\, \d r \\ &  \le \Con \e^2 \int_{\e^{a}}^{\e^{1/4}}\int_{|y|\ge 1} [(r+\e)^{-1/2}p_{r+\e}(y/2)+r^{-1/2}p_{r}(y/2)] \d y\, \d r \\ & \le \Con \e^2\int_{\e^{a}}^{\e^{1/4}}[(r+\e)^{-1/2}+r^{-1/2}] \,\d r \\ & \le \Con \e^2\int_{0}^{\e^{1/4}} r^{-\frac12} \,\d r \le \Con \e^{2+1/8}.
	\end{align*}}
	Combining the above estimate along with the estimate for the $\{|y|\le 1\}$ integral, we get $\Upsilon([t-\e^{1/4},t-\e^a)) \le \Con \e^{2-\frac{3}2\min(a,1)},$ completing the proof.
\end{proof}


\section{Comparison to Fractional Brownian motion} \label{sec:fbm}

\subsection{Proof of Theorem \ref{t:comp}}
The goal of this section is to prove Theorem \ref{t:comp} which compares increments of \ac{SHE} with that of fractional Brownian motion. Towards this end, we first recall the definition of fractional Brownian motion.

\begin{definition} \label{fbm}
	A standard fractional Brownian motion (fBm($H$)) with Hurst parameter $H\in (0,1)$ is a continuous, mean-zero Gaussian process $X := \{X_t\}_{t>0}$ with $X_0=0$ a.s.~and
	\begin{align}\label{e:fbm}
		\Ex(|X_t-X_s|^2)=|t-s|^{2H}.
	\end{align}
\end{definition}
Note that $H=\frac12$ corresponds to a standard Brownian motion. In this article we will be only dealing with $H=\frac14$ fractional Brownian motion. We will reserve the notation $\B_t$ for fBM($1/4$). 

\medskip

Coming back to the proof of Theorem \ref{t:comp}, we consider another space-time stochastic process defined via \begin{align}\label{vhe}
	\calV_t(x):=\int_{(0,t)\times \R} p_{t-s}(y-x)\xi(\d s,\d y)
\end{align}
with the same underlying space-time white noise $\xi(t,x)$.  $\calV_t(x)$ solves the linear stochastic heat equation $\partial_t\calV=\frac12\partial_{xx}\calV+\xi$ with $\calV_0(x)\equiv 0$.

As outlined in the introduction, we show that the temporal increments of the non-linear stochastic heat equation \ac{SHE} (i.e., \eqref{she}) can be related to temporal increments of the linear one via the following proposition.

\begin{proposition} \label{p:comp} Fix $\mathfrak{q} \in ((0,\infty)\times (0,1])^{\mathbb{N}}$, $T>1$, $\beta\in (0,\frac25)$ and $k\ge 2$. Consider the \ac{SHE} defined in \eqref{she} with initial data $f$ in $\hypz$ class defined in Definition \ref{hyp}. There exist constants $\e_0=\e_0(\mathfrak{q}_k,k,\beta,T)>0$ and $\Con=\Con(\mathfrak{q}_k,k,\beta,T)>0$, such that for all $\e\in (0,\e_0)$ we have
	\begin{align}\label{mombd2}
		\sup_{t\in [T^{-1},T]}\Ex\left[|(\calZ_{t+\e}^f-\calZ_t^f)-\calZ_t^f(\calV_{t+\e}-\calV_t)|^k\right] \le \Con_{k,T} \e^{k\beta},
	\end{align}
	where $\calV_s:=\calV_s(0)$ defined in \eqref{vhe}. 
\end{proposition}

Proof of Theorem \ref{t:comp} can be completed assuming Proposition \ref{p:comp}.

\begin{proof}[Proof of Theorem \ref{t:comp}]
	By Proposition 2.2 in \cite{davar} ($\alpha=2$ in their statement), one can find a fractional Brownian motion $\B_t$ with Hurst parameter $\frac14$ and another process $R_t$ on the same probability space as $\calV_t$ such that  $$\calV_t=(2/\pi)^{\frac14}\B_t+R_t.$$ 
 {By Minkowski's inequality followed by Cauchy-Schwarz inequality we have
 \begin{align*}
     & \Ex\left[|(\calZ_{t+\e}^f-\calZ_t^f)-(2/\pi)^{1/4}\calZ_t^f(\mathfrak{B}_{t+\e}-\mathfrak{B}_t)|^k\right]^{\frac1k} \\ & \le \Ex\left[|(\calZ_{t+\e}^f-\calZ_t^f)-\calZ_t^f(\calV_{t+\e}-\calV_t)|^k\right]^{\frac1k}+\Ex\left[|\calZ_t^f(R_{t+\e}-R_t)|^k\right]^{\frac1k} \\ & \le \Ex\left[|(\calZ_{t+\e}^f-\calZ_t^f)-\calZ_t^f(\calV_{t+\e}-\calV_t)|^k\right]^{\frac1k}+\Ex[|\calZ_t^f|^{2k}]^{\frac1{2k}}\Ex\left[|R_{t+\e}-R_t|^{2k}\right]^{\frac1{2k}}.
 \end{align*}
 Thanks to \eqref{mombd2}, the first term on the r.h.s.~of the above equation is  at most $\Con\e^{\beta}$ for some $\Con>0$ depending on $k,T$.
 According to Remark 2.4 in \cite{davar}, the process $R_t$ satisfies $\Ex[|R_{t+\e}-R_{t}|^{2k}] \le \Con_{K,T}\cdot\e^{2k}$ for all $t\in [T^{-1},T]$, $k\ge 1$ and $\e\in (0,1)$.} By Propositions \ref{umom} and \ref{umomf}, $\Ex[[\calZ_t^f]^{2k}] \le \Con$ for all $t\in [T^{-1},T]$ for functional initial data and narrow wedge initial data respectively. Combining all the estimates leads to the desired result.
\end{proof}

\begin{proof}[Proof of Proposition \ref{p:comp}] Recall the evolution equation of $\mathcal{Z}^f(t,x)$ from \eqref{fkdc} and definition of $\calV_t$ from \eqref{vhe}. At this point, we also invite the reader to review the proof strategy explained earlier in Section \ref{sec:pfidea}.

Fix any $a\ge 1/2$ and $\e \in (0,\frac1{4T^2})$. Since $t\ge T^{-1}$, we have $\frac1{2T} \le t-\e^a$. A bit of algebra leads to the following decomposition:
{\begin{align}\label{zvdec}
		\calZ_{t+\e}^f-\calZ_{t}^f-\calZ_t^f(\calV_{t+\e}-\calV_t) & =  \int_{\R} \left[p_{t+\e}(y)-p_t(y)\right]e^{f(y)}\d y +\Phi_1-\Phi_2 +\Lambda_1+\Lambda_2+\mathsf{Rem},
	\end{align}
	where  
	\begin{align}
            \Phi_1 & :=\int_{0}^{t-\e^a}\int_{\R}\left[p_{t+\e-s}(y)-p_{t-s}(y)\right]\calZ_s^f(y)\xi(\d s,\d y), \label{phi1} \\
            \Phi_2 & :=\calZ_t^f\int_{0}^{t-\e^a}\int_{\R}\left[p_{t+\e-s}(y)-p_{t-s}(y)\right]\xi(\d s,\d y), \label{phi2} \\
		\Lambda_1 & :=\int_{(t-\e^a,t)\times \R}\left[p_{t+\e-s}(y)-p_{t-s}(y)\right]\left[\calZ_s^f(y)-\calZ_{t-\e^a}^f\right]\xi(\d s,\d y), \label{lam1} \\
            \Lambda_2 & :=(\calZ_{t-\e^a}^f-\calZ_t^f)\int_{(t-\e^a,t)\times \R}\left[p_{t+\e-s}(y)-p_{t-s}(y)\right]\xi(\d s,\d y), \label{lam2} \\
		\mathsf{Rem} & :=\int_{(t,t+\e)\times \R}p_{t+\e-s}(y)\left[\calZ_s^f(y)-\calZ_t^f\right]\xi(\d s,\d y). \label{rem}
	\end{align}}
{Given this decomposition, using the elementary inequality $(\sum_{m=1}^n a_m)^k \le n^{k-1}\sum_{m=1}^n a_m^k$ we get that
\begin{equation}
    \label{momdn}
    \begin{aligned}
    & \Ex[|\calZ_{t+\e}^f-\calZ_{t}^f-\calZ_t^f(\calV_{t+\e}-\calV_t)|^k] \\ & \hspace{1cm}\le \Con \cdot \Ex\bigg[\big|\int_{\R} \left[p_{t+\e}(y)-p_t(y)\right]e^{f(y)}\d y\big|^k\bigg] \\ &  \hspace{2cm} + \Con\left[\Ex[|\Phi_1|^k]+\Ex[|\Phi_2|^k] +\Ex[|\Lambda_1|^k]+\Ex[|\Lambda_2|^k]+\Ex[|\mathsf{Rem}|^k]\right].
\end{aligned}
\end{equation}
}
{It is now suffices to estimates the moments of each term appearing on the r.h.s.~of \eqref{zvdec}.}

{For the first term, from \eqref{claimg} we know that
\begin{align} \label{claimgg}
    \Ex\bigg[\big|\int_{\R} \left[p_{t+\e}(y)-p_t(y)\right]e^{f(y)}\d y\big|^k\bigg] \le \Con \cdot \e^k.
\end{align}
For the $\Phi_i$ terms, we claim that there exist constants $\e_0=\e_0(\mathfrak{q}_k,k,T)>0$, $\Con=\Con(\mathfrak{q}_k,k,T)>0$, such that for all $t\in [T^{-1},T]$, $\e\in (0,\e_0)$, and $f\in \hypz$ we have
 \begin{align}\label{claimh}
   \Ex[|\Phi_1|^k+|\Phi_2|^k] \le \Con\cdot\e^{k(1-\frac{3a}{4})}.  
 \end{align}
 Note that by BDG inequality (Theorem \ref{thm9}) and the integral estimate from Lemma \ref{heat2}, we get that
 \begin{align*}
     \Ex[|\Phi_1|^k]^{\frac2k} \le \Con\int_{0}^{t-\e^a}\int_{\R}\left[p_{t+\e-s}(y)-p_{t-s}(y)\right]^2\Ex[|\calZ_s^f(y)|^k]^{\frac2k}\,\d y\,\d s \le \Con e^{2-\frac32a}.
 \end{align*}
	For $\Phi_2$ we apply H\"older's inequality to find that
 \begin{align*}
     \Ex[|\Phi_2|^k]^{\frac2k} & \le \Ex[|\calZ_t^f|^{2k}]^{\frac1k}\cdot \Ex\left[\left|\int_{0}^{t-\e^a}\int_{\R}\left[p_{t+\e-s}(y)-p_{t-s}(y)\right]\xi(\d s,\d y)\right|^{2k}\right]^{\frac1k} \\ & \le \Con \cdot \Ex[|\calZ_t^f|^{2k}]^{\frac1k}\cdot \int_{0}^{t-\e^a}\int_{\R}\left[p_{t+\e-s}(y)-p_{t-s}(y)\right]^2\,\d y\,\d s,
 \end{align*}
where the last inequality follows from another application of BDG inequality (Theorem \ref{thm9}).	Using the moment estimate from Proposition \ref{umomf} and the integral estimate from Lemma \ref{heat2} we see that the last expression above is at most $\Con e^{2-\frac32a}$. Combining this with the moment estimate of $\Phi_1$ obtained previously, verifies \eqref{claimh}.} 

We have the following moment estimates for $\m{Rem}$ defined in \eqref{rem}.
	\begin{lemma}\label{l4} Fix $\mathfrak{q} \in ((0,\infty)\times (0,1])^{\mathbb{N}}$, $T>1$ and $k\ge 2$. Fix $\theta\in (0,\frac12)$. There exist constants $\e_0=\e_0(\mathfrak{q}_k,k,T,\theta)>0$, $\Con=\Con(\mathfrak{q}_k,k,T,\theta)>0$, such that for all $t\in [T^{-1},T]$, $\e\in (0,\e_0)$, and $f\in \hypz$ we have
		$$\Ex[|\mathsf{Rem}|^k] \le \Con\cdot\e^{k(\frac\theta2+\frac14)}.$$
	\end{lemma}
	   We have the following moment estimates for $\Lambda_1$ and $\Lambda_2$ defined in \eqref{lam1} and \eqref{lam2}.
	\begin{lemma}\label{l3} Fix $\mathfrak{q} \in ((0,\infty)\times (0,1])^{\mathbb{N}}$, $T>1$ and $k\ge 2$. Fix $a\in (\frac12,1)$ and $\theta\in (\frac12,1)$. There exist constants $\e_0=\e_0(\mathfrak{q}_k,k,T,a,\theta)>0$, $\Con=\Con(\mathfrak{q}_k,k,T,a,\theta)>0$, such that for all $t\in [T^{-1},T]$, $\e\in (0,\e_0)$, and $f\in \hypz$ we have  $$\Ex[|\Lambda_1|^k] \le \Con\cdot e^{ka(\frac\theta2+\frac14)}, \qquad \Ex[|\Lambda_2|^k] \le \Con\cdot e^{ka/2}.$$ 
	\end{lemma}
	
	We defer the proof of the above two lemmas to next subsection. Let us complete the proof of Proposition \ref{p:comp} assuming it. Fix any $\beta\in (0,\frac25)$. Choose $\theta$ such that $\frac{1+2\theta}{4+2\theta}=\beta$. The fact that $\beta<\frac25$ ensures $\theta<\frac12$. Set $a=\frac{2}{2+\theta}$ so that $\frac{a(2\theta+1)}4=1-\frac{3a}4=\frac{1+2\theta}{4+2\theta}=\beta$. Combining all the moment estimates from \eqref{claimgg}, \eqref{claimh}, Lemma \ref{l4}, and \ref{l3}, for this choice of $a$ and $\theta$, in view of the bound in \eqref{momdn}, we arrive at \eqref{mombd2}.
\end{proof}

\subsection{Proof of auxiliary lemmas} In this subsection we prove Lemma \ref{l4} and \ref{l3}. We fix $T>1$, $k\ge 2$, $a\in (\frac12,1)$, $\theta\in (0,\frac12)$ and $\gamma\in (0,\frac14)$. In all the proofs below the constants $\Con, \e_0$ depends on $a$, $k$, $T$, $\gamma$, $\theta$, and $f$ via $\lambda_k$ and $\delta_k$ defined in \eqref{mkcond}. We won't mention this further. 

\begin{proof}[Proof of Lemma \ref{l4}]
	For the remainder term $\mathsf{Rem}$ defined in \eqref{rem}, we rewrite it as $\mathsf{Rem}_{1}+\mathsf{Rem}_{2}$ where
	\begin{align*}
		\mathsf{Rem}_{1} & :=\int_{(t,t+\e)\times \R}p_{t+\e-s}(y)\left[\calZ_s^f(y)-\calZ_s^f(0)\right]\xi(\d s,\d y), \\ \mathsf{Rem}_{2} & :=\int_{(t,t+\e)\times \R}p_{t+\e-s}(y)\left[\calZ_s^f(0)-\calZ_t^f(0)\right]\xi(\d s,\d y).
	\end{align*}
	Applying Minkowski's inequality followed by BDG inequality (Theorem \ref{thm9}) separately for $\mathsf{Rem}_{1}$ and $\mathsf{Rem}_{2}$ we get
	\begin{align} \nonumber
		\Ex[|\mathsf{Rem}|^k]^{\frac2k} & \le \Con\left[\Ex[|\mathsf{Rem}_{1}|^k]^{\frac2k}+\Ex[|\mathsf{Rem}_{2}|^k]^{\frac2k}\right] \\ & \le \Con\int_t^{t+\e}\int_{\R}p^2_{t+\e-s}(y)\Ex\left[\big|\calZ_s^f(y)-\calZ_s^f(0)\big|^k\right]^{\frac2k}\,\d y\,\d s \label{bdg0} \\ & \hspace{2cm}+\Con\int_t^{t+\e}\int_{\R}p^2_{t+\e-s}(y)\Ex\left[\big|\calZ_s^f(0)-\calZ_t^f(0)\big|^k\right]^{\frac2k} \,\d y\,\d s\label{bdg1}.
	\end{align}
	For the expectation in the integrand of the term in \eqref{bdg0} we use Proposition \ref{spmomf} for small enough $|y|$. For large $|y|$, we use the generic moment estimate from Proposition \ref{umomf}. Together we have 
	\begin{align}\label{d3}
		\mbox{r.h.s.~of \eqref{bdg0}} \le \Con\int_t^{t+\e}\int_{\R}p^2_{t+\e-s}(y)\min\{ |y|^{2\theta}, e^{\Delta |y|^{2-\delta_k}}\}\,\d y\,\d s.
	\end{align}
	Taking $\rho=\e$ in the Lemma \ref{heat} we see that r.h.s.~of \eqref{d3} is at most $\Con\cdot \e^{\frac12+\theta}$. 
	For the expectation in the integrand of the term in \eqref{bdg1} we use Proposition \ref{tpmomf} to get
	\begin{align}\label{d4}
		\mbox{r.h.s.~of \eqref{bdg1}} \le \Con\int_t^{t+\e}\int_{\R}p^2_{t+\e-s}(y)|s-t|^{\frac12}\,\d y\,\d s = \Con\int_t^{t+\e}\tfrac1{2\sqrt{\pi(t+\e-s)}}|s-t|^{\frac12}\,\d s. 
	\end{align}
	Making a change of variable $x\mapsto \e^{-1}(s-t)$ we get 
	\begin{align*}
		\mbox{r.h.s.~of \eqref{d4}}  =\Con\tfrac1{2\sqrt\pi}\e\int_0^{1}x^{\frac12}(1-x)^{-\frac12}\d x \le \Con'\cdot \e.
	\end{align*} 
	Combining this with the estimate of the double integral in \eqref{bdg0} proves Lemma \ref{l4}.
\end{proof}

\begin{proof}[Proof of Lemma \ref{l3}] {Recall the expression of $\Lambda_1$ from \eqref{lam1}. Set
\begin{align*}
    \Lambda_{1,1} & :=\int_{(t-\e^a,t)\times \R}\left[p_{t+\e-s}(y)-p_{t-s}(y)\right]\left[\calZ_s^f(y)-\calZ_{s}^f(0)\right]\xi(\d s,\d y), \\
     \Lambda_{1,2} & :=\int_{(t-\e^a,t)\times \R}\left[p_{t+\e-s}(y)-p_{t-s}(y)\right]\left[\calZ_s^f(0)-\calZ_{t-\e^a}^f(0)\right]\xi(\d s,\d y),
\end{align*}
so that $\Lambda_1=\Lambda_{1,1}+\Lambda_{1,2}$
Since $\calZ_{t-\e^a}^f(0)$  is measurable w.r.t.~white noise on $[0,t-\e^a]\times \R$, it is independent of the white noise of $(t-\e^a,t)\times \R$. By Minkowski's inequality followed by an application of the BDG inequality (Theorem \ref{thm9}) separately for $\Lambda_{1,1}$ and $\Lambda_{1,2}$ to get that 
\begin{align} \nonumber	\left[\Ex\big[\left|\Lambda_1\right|^k\big]\right]^{\frac2k} & \le \Con \left[\left[\Ex\big[\left|\Lambda_{1,1}\right|^k\big]\right]^{\frac2k}+\left[\Ex\big[\left|\Lambda_{1,2}\right|^k\big]\right]^{\frac2k}\right] \\ & \le \Con\int_{t-\e^a}^{t}\int_{\R}[p_{t+\e-s}(y)-p_{t-s}(y)]^2\Ex[\left|\calZ_s^f(y)-\calZ_s^f(0)\right|^k]^{\frac2k}\,\d y\,\d s \\ \nonumber & \hspace{1cm}+\Con\int_{t-\e^a}^{t}\int_{\R}[p_{t+\e-s}(y)-p_{t-s}(y)]^2\Ex[\left|\calZ_s^f(0)-\calZ_{t-\e^a}^f(0)\right|^k]^{\frac2k}\,\d y\,\d s \\ & \label{1d} \le \Con{\int_{t-\e^a}^{t}\int_{\R}[p_{t+\e-s}(y)-p_{t-s}(y)]^2\min\{ |y|^{2\theta}, e^{\Delta |y|^{2-\delta_k}}\}\,\d y\,\d s} \\ & \hspace{2cm}+\Con\int_{t-\e^a}^{t}\int_{\R}[p_{t+\e-s}(y)-p_{t-s}(y)]^2|s-t+\e^a|^{\frac12}\,\d y\,\d s, \label{2d}		
	\end{align}
where the last inequality follows by applying the moment estimate and spatial regularity estimate  from Propositions \ref{umomf} and \ref{spmomf} and temporal regularity estimate from Proposition \ref{tpmomf}. For the double integral in \eqref{1d} we apply $[p_{t+\e-s}(y)-p_{t-s}(y)]^2 \le p_{t+\e-s}(y)^2+2p_{t-s}(y)^2$ and separate it into two double integrals. Using Lemma \ref{heat} with $\rho\mapsto \e^a+\e$, we get that each of the double integrals is at most $\Con \e^{a(\frac12+\theta)}$. For the double integral in \eqref{2d}, we can compute the $y$ integral explicitly. It is straightforward to check that
	\begin{align}\label{explicity}
		\int_{\R}[p_{t+\e-s}(y)-p_{t-s}(y)]^2\d y = \tfrac1{2\sqrt\pi}\left[\tfrac1{\sqrt{t+\e-s}}+\tfrac1{\sqrt{t-s}}-\tfrac2{\sqrt{t-s+\frac{\e}2}}\right] \le \tfrac1{\sqrt{\pi(t-s)}}.
	\end{align}
	Using the above bound along with a change of variable $r\mapsto t-s$ we get $\mbox{r.h.s.~of \eqref{2d}} \le \Con\int_{0}^{\e^a} r^{-1/2}(\e^a-r)^{1/2}\,\d r  \le \Con \e^{a}$. Combining this with the estimate for the double integral in \eqref{1d} we get the desired moment estimate for $\Lambda_1$.}

{For $\Lambda_2$ we apply H\"older's inequality to find that
 \begin{align*}
     \Ex[|\Lambda_2|^k]^{\frac2k} & \le \Ex[|\calZ_{t-\e^a}^f-\calZ_t^f|^{2k}]^{\frac1k}\cdot \Ex\left[\left|\int_{t-\e^a}^t\int_{\R}\left[p_{t+\e-s}(y)-p_{t-s}(y)\right]\xi(\d s,\d y)\right|\right]^{\frac1k} \\ & \le \Con \cdot \Ex[|\calZ_{t-\e^a}^f-\calZ_t^f|^{2k}]^{\frac1k}\cdot \int_{t-\e^a}^t\int_{\R}\left[p_{t+\e-s}(y)-p_{t-s}(y)\right]^2\,\d y\,\d s, \\ & \le \Con \cdot \Ex[|\calZ_{t-\e^a}^f-\calZ_t^f|^{2k}]^{\frac1k}\cdot \int_{t-\e^a}^t \frac{1}{\sqrt{\pi(t-s)}}\d s.
 \end{align*}
The inequality on the second line above follows from another application of BDG inequality (Theorem \ref{thm9}). The third line inequality follows from \eqref{explicity}. Evaluating the integral in the third line and using the temporal regularity estimate from Proposition \ref{tpmomf} shows that the last expression is at most $\Con \cdot \e^a$. This leads to the desired moment estimates for $\Ex[|\Lambda_2|^k]$.} 
\end{proof}

\section{Proof of Theorem \ref{t:main} and \ref{t:cor}} \label{sec:pf}
The goal of this section is to prove Theorem \ref{t:main}. We first state one point tail estimates for the unscaled KPZ equation that allow us to control negative moments of Stochastic Heat Equation as well. 

\begin{proposition}\label{hftail} Fix $T>1$ and $\mathfrak{p}\in (0,\infty)^5$. There exists a constant $\Con=\Con(\mathfrak{p},T)>0$ such that for all $t\in [T^{-1},T]$ and for all $f\in \hyph$ we have 
	\begin{align*}
		\Pr(|\calH_t^f| \ge s) \le \Con\exp(-\tfrac1\Con s^{3/2}).
	\end{align*}
\end{proposition}
\begin{proof} Fix $T>1$ and $\mathfrak{p}\in (0,\infty)^5$. Consider the scaled KPZ equation $$\h_t^f:=t^{-1/3}(\calH_{2t}^f+\tfrac{t}{12}-\tfrac23\log(2t)).$$ Since $t\in [T^{-1},T]$ and we allow $\Con$ to be dependent on $T$, it is suffices to show $\Pr(|\h_t^f| \ge s) \le \Con\exp(-\tfrac1\Con s^{3/2})$. This is essentially proved in \cite{kpzgen} for narrow wedge initial data, Brownian initial data, and for certain functional initial data belonging to $\mathbf{Hyp}$ class in their Definition 1.1. Their $\mathbf{Hyp}$ class is adapted to the KPZ fixed point scaling and consists of conditions on $g(y):=t^{-1/3}f((2t)^{2/3}y)$. We briefly explain how our $\hyph$ class is contained in their class of initial data. Assume $f$ to be functional initial data. Recall the condition \ref{ci} from Definition \ref{hyp1}.	As $t\in [T^{-1},T]$ and $\delta >0$, one can get a constant $A=A(\lambda,\delta, T)$ so that  
	$$g(y) \le \lambda t^{-1/3}(1+|(2t)^{2/3}y|)^{2-\delta}\le A+\tfrac{1}{2\cdot2^{2/3}}y^2.$$
	This verifies Eq.~(1.2) in \cite{kpzgen} for $C\mapsto A$ and $\nu\mapsto \frac12$. 
	Furthermore due to condition \ref{cii}, there must exist an interval $\mathcal{I} \subset [-(2T)^{2/3}M,(2T)^{2/3}M]$ with size at least $(2T)^{2/3}\theta$ where $g(y)\ge -\kappa T^{-1/3}$. This verifies Eq.~(1.3) in \cite{kpzgen} for $\theta \mapsto (2T)^{2/3}\theta, \kappa \mapsto \kappa T^{-1/3},$ and $M \mapsto (2T)^{2/3}M.$
	
	\medskip
	
	Thus, $g\in \mathbf{Hyp}(A,\frac12,(2T)^{2/3}M,(2T)^{2/3}\theta,-
	\kappa T^{-1/3})$ according to Definition 1.1 in \cite{kpzgen}. Since we allow $\Con$ to be dependent on $T$, Theorem 1.2 and Theorem 1.4 in \cite{kpzgen} leads to the desired conclusion.
\end{proof}

\begin{proof}[Proof of Theorem \ref{t:main}] Fix $\mathfrak{p}\in (0,\infty)^5$ and fix any $f\in \hyph$. For simplicity, we drop the $f$ from the notation and write $\calH_t$ and $\calZ_t$ for $\calH_t^f$ and $\calZ_t^f$ respectively. We obtain a $\mathfrak{q} \in ((0,\infty)\times (0,1])^{\mathbb{N}}$ from Lemma \ref{gid} so that $\hyph \subset \hypz$. Given any $k\ge 2$, $\beta\in (0,\frac25)$ and $T>1$ we will show that there exists a probability space containing $\calH_t$ as well as a fractional Brownian motion $\B_t$ with index $\frac14$ and a constant $\Con>0$ depending on $k,\beta,T$ such that
	\begin{align}\label{e:mid}
		\sup_{s,t\in [T^{-1},T]} \Ex\left[\left|\calH_{s}-\calH_t-(2/\pi)^{1/4}(\B_{s}-\B_t)\right|^{k}\right] \le \Con |s-t|^{k\beta}.
	\end{align}
	Then Theorem \ref{t:main} follows from Kolmogorov's continuity theorem. 

 {Note that the one point exponential tail estimates from Proposition \ref{hftail} ensures moments of $\Ex[|\mathcal{H}_s|^k+|\mathcal{H}_t|^k]$ are uniformly bounded as $s,t$ varies on $[T^{-1},T]$. By properties of fractional Brownian motion, $\Ex[|\mathfrak{B}_s-\mathfrak{B}_t|^k]$ is uniformly bounded as $s,t$ varies on $[T^{-1},T]$. Thus by Minkowski's inequality the supremum on the left-hand side of \eqref{e:mid} is finite.} Thus, whenever $|s-t|$ is bounded away from zero, \eqref{e:mid} follows by choosing $\Con$ large enough. So, without loss of generality we may assume $s=t+\e$ with $\e\in (0,\e_0)$ where $\e_0$ is small enough. We divide the rest of the proof into two steps.
	
	\medskip
	
	\noindent\textbf{Step 1.} In this step we claim that
	\begin{align}\label{uint}
		\sup_{t\in [T^{-1},T]} \e^{-\frac{k}4}\Ex[\left|\calH_{t+\e}-\calH_t|^{k}\right] <\infty,
	\end{align}
	for all $\e\le \e_0$ where $\e_0=\e_0(\beta,\mathfrak{q},k,T)>0$ is chosen from Theorem \ref{t:comp}. Observe that
	$$|x|^{k} \le (e^{|x|}-1)^{k} \le |e^{-x}-1|^{k}+|e^{x}-1|^{k}.$$		Take $x=\calH_{t+\e}-\calH_{t}$ and then divide the aforementioned inequality by $\e^{\frac{k}4}$ to get
	\begin{align}
		\nonumber \e^{-\frac{k}4}|\calH_{t+\e}-\calH_t|^k & \le \calZ_{t+\e}^{-k} \cdot \e^{-\frac{k}4}|\calZ_{t+\e}-\calZ_{t}|^k+[\calZ_t]^{-k} \cdot \e^{-\frac{k}4}|\calZ_{t+\e}-\calZ_{t}|^k \\ & \nonumber \le \frac12\left([\calZ_{t+\e}]^{-2k}+\e^{-\frac{k}2}|\calZ_{t+\e}-\calZ_{t}|^{2k}\right)+ \frac12\left([\calZ_t]^{-2k}+\e^{-\frac{k}2}|\calZ_{t+\e}-\calZ_{t}|^{2k}\right) \\ & \label{in1} = \e^{-\frac{k}2}|\calZ_{t+\e}-\calZ_{t}|^{2k}+ \frac12\left([\calZ_{t+\e}]^{-2k}+[\calZ_t]^{-2k}\right), 
	\end{align}
	where the inequality in the second line is an application of the $ab \le \frac12(a^2+b^2)$ inequality applied with $a=\calZ_{t+\e}^{-k}$ and $b=\e^{-\frac{k}4}|\calZ_{t+\e}-\calZ_{t}|^k$ and with $a=\calZ_{t}^{-k}$ and $b=\e^{-\frac{k}4}|\calZ_{t+\e}-\calZ_{t}|^k$. Set \begin{align}\label{thetarv}
		\Theta_{t,\e}:=\calZ_{t+\e}-\calZ_t -(2/\pi)^{1/4}\calZ_t(\B_{t+\e}-\B_t),
	\end{align} 
 so that $\calZ_{t+\e}-\calZ_t =\Theta_{t,\e}+(2/\pi)^{\frac14}\calZ_t(\B_{t+\e}-\B_t)$. Using the elementary inequality $|a+b|^{2k} \le 2^{2k-1}(a^{2k}+b^{2k})$ with $a=\Theta_{t,\e}$ and $b=(2/\pi)^{1/4}\calZ_t(\B_{t+\e}-\B_t)$ we get that
 \begin{align*}
     |\calZ_{t+\e}-\calZ_{t}|^{2k} \le \Con \left[|\Theta_{t,\e}|^{2k}+\calZ_t^{2k}|\B_{t+\e}-\B_{t}|^{2k}\right].
 \end{align*}
 Plugging this inequality back in r.h.s.~of \eqref{in1} we get
	\begin{align*}
		\mbox{r.h.s.~of \eqref{in1}} & \le \Con\e^{-\frac{k}2}\left[|\Theta_{t,\e}|^{2k}+\calZ_t^{2k}|\B_{t+\e}-\B_{t}|^{2k}\right]+ \frac12\left(\calZ_{t+\e}^{-2k}+\calZ_t^{-2k}\right) \\ & \le \Con\e^{-\frac{k}2}|\Theta_{t,\e}|^{2k}+\Con\left[\calZ_t^{4k}+\e^{-k}|\B_{t+\e}-\B_{t}|^{4k}\right]+ \frac12\left(\calZ_{t+\e}^{-2k}+\calZ_t^{-2k}\right)
	\end{align*}
	where the last inequality is another application of $ab \le \frac12(a^2+b^2)$ inequality with $a=\mathcal{Z}_t^{2k}$ and $b=\e^{-k/2}|\mathfrak{B}_{t+\e}-\mathfrak{B}_t|^{2k}$. Taking expectations and utilizing \eqref{mombd} with $\beta=\frac13<\frac25$ (for $\e\le \e_0$) and the fact that $\e^{-\frac14}(\B_{t+\e}-\B_t)\stackrel{d}{=}Z \sim N(0,1)$ we get
	\begin{align*}
		\sup_{t\in [T^{-1},T]} \e^{-\frac{k}4}\Ex[\left|\calH_{t+\e}-\calH_t|^{k}\right] \le \Con\left[\e^{-\frac{k}2+\frac{2k}3}+\Ex[|\calZ_t|^{4k}]+\sup_{t\in [T^{-1},T+1]} \Ex\left([\calZ_t]^{-2k}+[\calZ_t]^{4k}\right)\right],
	\end{align*}
	for all $\e\in (0,\e_0)$ and for some constant $\Con>0$ depending on $\mathfrak{q}_k,k,T$. The moments of $\mathcal{Z}_t$ (both negative and positive) can be bounded uniformly on $t\in [T^{-1},T]$ using the uniform one point upper and lower tail bounds for $\calH_{t}=\log \calZ_{t}$ for $t\in [T^{-1},T]$ from Proposition \ref{hftail}. Thus the right hand side of above equation is finite uniformly in $\e \in (0,\e_0)$. This completes our work in this step.
	
	\medskip
	
	\noindent\textbf{Step 2.} In this step we show that
	\begin{align}\label{uint2}
		\sup_{t\in [T^{-1},T]}\e^{-\frac{k}2}\Ex\left[\left|\frac{\calZ_{t+\e}}{\calZ_{t}}-1-\calH_{t+\e}+\calH_t\right|^k\right] < \infty, 
	\end{align}
	for all $\e\le \e_0$ where $\e_0=\e_0(\beta,\mathfrak{q},2k,T)>0$ is chosen from Theorem \ref{t:comp}. Assuming \eqref{uint2}, proof of \eqref{e:mid} for $s=t+\e$ can be shown as follows. 
 
{First observe that
\begin{equation}
\label{ident2}
     \begin{aligned}
    & \e^{-\beta}\left[\calH_{t+\e}-\calH_t-(2/\pi)^{\frac14}(\B_{t+\e}-\B_t)\right] \\ & = \e^{-\beta}\left[\calH_{t+\e}-\calH_t -\frac{\calZ_{t+\e}}{\calZ_{t}}+1\right]+\e^{-\beta}\left[\frac{\calZ_{t+\e}}{\calZ_{t}}-1-(2/\pi)^{\frac14}(\B_{t+\e}-\B_t)\right] \\ & = \e^{-\beta}\left[\calH_{t+\e}-\calH_t -\frac{\calZ_{t+\e}}{\calZ_{t}}+1\right]+\calZ_t^{-1}\cdot \e^{-\beta}\Theta_{t,\e}
 \end{aligned}
\end{equation}
 where $\Theta_{t,\e}$ is defined in \eqref{thetarv}. Applying the elementary inequality $|a+bc|^k \le 2^{k-1}|a|^k +2^{k-2}b^{2k}+2^{k-2}c^{2k}$ with 
 \begin{align*}
     a= \e^{-\beta}\left[\calH_{t+\e}-\calH_t -\frac{\calZ_{t+\e}}{\calZ_{t}}+1\right], \quad b=\calZ_t^{-1}, \quad c=\e^{-\beta}\Theta_{t,\e},
 \end{align*}
 in view of the identity in \eqref{ident2}, we get that
	\begin{align*}
		\e^{-k\beta}\Ex\left[\big|\calH_{t+\e}-\calH_t-(2/\pi)^{1/4}(\B_{t+\e}-\B_t)\big|^k\right] & \le \Con\e^{-k\beta}\Ex\left[\big|\frac{\calZ_{t+\e}}{\calZ_{t}}-1-\calH_{t+\e}+\calH_t\big|^k\right] \\ & \hspace{1cm}+\Con\Ex[\calZ_t^{-2k}]+ \Con\e^{-2k\beta}\Ex[|\Theta_{t,\e}|^{2k}].
	\end{align*}}
	 Each of the three expectations appearing on r.h.s.~of the above equation is uniformly bounded by \eqref{uint2}, tail estimates from Proposition \ref{hftail} and Theorem \ref{t:comp} respectively. This verifies \eqref{e:mid} modulo \eqref{uint2}. Hence it suffices to show \eqref{uint2}. Towards this end, first observe that
	$$|e^x-1-x|^k \le x^{2k}e^{k|x|} \le x^{2k}(e^{-x}+e^x)^k.$$
	Take $x=\calH_{t+\e}-\calH_{t}$ and then divide the aforementioned inequality by $\e^{\frac{k}2}$ to get
	\begin{align}
		\label{inde} \e^{-\frac{k}2}\left|\frac{\calZ_{t+\e}}{\calZ_{t}}-1-\calH_{t+\e}+\calH_t\right|^k & \le \e^{-\frac{k}2}\left|\calH_{t+\e}-\calH_t\right|^{2k} \left|\frac{\calZ_{t}}{\calZ_{t+\e}}+\frac{\calZ_{t+\e}}{\calZ_{t}}\right|^k \\ & \le  \e^{-\frac{k}2}\left|\calH_{t+\e}-\calH_t\right|^{2k}\calZ_{t+\e}^{-k}\calZ_t^{-k}|\calZ_t^2+\calZ_{t+\e}^2|^k.  \nonumber
	\end{align}
	Using the AM-GM inequality we may bound the product of above four terms by sum of fourth powers of each term. This gives us
	\begin{align*}
		\mbox{r.h.s.~of \eqref{inde}} & \le \Con \left[\e^{-2k}\left|\calH_{t+\e}-\calH_t\right|^{8k}+\calZ_{t+\e}^{-4k}+\calZ_{t}^{-4k}+|\calZ_{t}^{2}+\calZ_{t+\e}^2|^{4k}\right] \\ & \le \Con \left[\e^{-2k}\left|\calH_{t+\e}-\calH_t\right|^{8k}+\calZ_{t+\e}^{-4k}+\calZ_{t}^{-4k}+\calZ_{t}^{8k}+\calZ_{t+\e}^{8k}\right],
	\end{align*}
where the last inequality follows by noting $|a^2+b^2|^{4k} \le 2^{4k-1}(a^{8k}+b^{8k})$ (with $a=\calZ_{t}^{2}$ and $b=\calZ_{t+\e}^{2})$. 
	The expectation of the first term on the right hand side of the above equation is uniformly bounded via  \eqref{uint}. The moments of $\mathcal{Z}_r$ (both negative and positive) can be bounded uniformly on $r\in [T^{-1},T]$ using the uniform one point upper and lower tail bounds for $\calH_{r}=\log \calZ_{r}$ for $r\in [T^{-1},T]$ from Proposition \ref{hftail}. Thus the expectation of the right hand side of the above equation is bounded uniformly in $\e \in (0,\e_0)$ and $t\in [T^{-1},T]$. This completes the proof.
\end{proof}

\begin{proof}[Proof of Theorem \ref{t:cor}] We appeal to Theorem \ref{t:main} to get a probability space containing the KPZ temporal process $\calH_t^f$ and fractional Brownian motion $\B_t$ such that $\calH_t^f-(2/\pi)^{1/4}\B_t$ is $\frac25^-$ H\"older continuous. Hence it is enough to understand the increment properties of $\B_{t}$ instead. Using \eqref{e:fbm} it follows that for $t_i\neq t_j$ the pairwise covariance $$\e^{-1/2}\operatorname{Cov}(\B_{t_i+\e}-\B_{t_i},\B_{t_j+\e}-\B_{t_j})\to 0$$  as $\e\downarrow 0$.  This leads to the proof of Part (1) of Theorem \ref{t:cor}. The quartic variation and the Law of iterated logarithm for $\B_t$ follows from Eq 1 in \cite{bquart} (see also \cite{bq1}) and Theorem 1.1 (C) in \cite{blil} respectively. Modulus of continuity and Hausdorff dimension for $\B_t$ follows from Theorem 2.1 and Theorem 5.3 in \cite{khos}.
\end{proof}



\bibliographystyle{alpha}

\begin{thebibliography}{CHHM21}
	
	\bibitem[1]{acq}
	G.~Amir, I.~Corwin, and J.~Quastel.
	\newblock Probability distribution of the free energy of the continuum directed
	random polymer in 1+ 1 dimensions.
	\newblock {\em Communications on pure and applied mathematics}, 64(4):466--537,
	2011.
	\MR{2796514}
	
	\bibitem[2]{bertini1995stochastic}
	L.~Bertini and N.~Cancrini.
	\newblock The stochastic heat equation: {F}eynman{--K}ac formula and
	intermittence.
	\newblock {\em J. Stat. Phys.}, 78(5-6):1377--1401, 1995.
	\MR{1316109}
	
	\bibitem[3]{BC16}
	R.~M. Balan and D.~Conus.
	\newblock Intermittency for the wave and heat equations with fractional noise
	in time.
	\newblock {\em Ann. Probab.}, 44(2):1488--1534, 2016.
	\MR{3474476}
	
	\bibitem[4]{bgh1}
	R.~Basu, S.~Ganguly, and A.~Hammond.
	\newblock Fractal geometry of airy $ \_ $\{$2$\}$ $ processes coupled via the
	airy sheet.
	\newblock {\em The Annals of Probability}, 49(1):485--505, 2021.
	\MR{4203343}
	
	\bibitem[5]{bgh2}
	E.~Bates, S.~Ganguly, and A.~Hammond.
	\newblock Hausdorff dimensions for shared endpoints of disjoint geodesics in
	the directed landscape.
	\newblock {\em Electronic Journal of Probability}, 27:1--44, 2022.
	\MR{4361743}
	
	\bibitem[6]{CD15}
	L.~Chen and R.~C. Dalang.
	\newblock Moments and growth indices for the nonlinear stochastic heat equation
	with rough initial conditions.
	\newblock {\em Ann. Probab.}, 43(6):3006--3051, 11 2015.
	\MR{3433576}
	
	\bibitem[7]{kpzgen}
	I.~Corwin and P.~Ghosal.
	\newblock {KPZ} equation tails for general initial data.
	\newblock {\em Electronic Journal of Probability}, 25:1--38, 2020.
	\MR{4115735}
	
	\bibitem[8]{kpztime}
	I.~Corwin, P.~Ghosal, and A.~Hammond.
	\newblock {KPZ} equation correlations in time.
	\newblock {\em The Annals of Probability}, 49(2):832--876, 2021.
	\MR{4255132}
	
	\bibitem[9]{ch16}
	I.~Corwin and A.~Hammond.
	\newblock {KPZ} line ensemble.
	\newblock {\em Probability Theory and Related Fields}, 166(1):67--185, 2016.
	\MR{3547737}
	
	\bibitem[10]{Ch17}
	L.~Chen.
	\newblock Nonlinear stochastic time-fractional diffusion equations on
	{$\mathbb{R}$}: moments, {H}\"{o}lder regularity and intermittency.
	\newblock {\em Trans. Amer. Math. Soc.}, 369(12):8497--8535, 2017.
	\MR{3710633}
	
	\bibitem[11]{chhm21}
	I.~Corwin, A.~Hammond, M.~Hegde, and K.~Matetski.
	\newblock Exceptional times when the {KPZ} fixed point violates johansson's
	conjecture on maximizer uniqueness.
	\newblock {\em Electronic Journal of Probability}, 28, pp.1-81, 2023.
	\MR{4536117}
	\bibitem[12]{CHN19}
	L.~Chen, Y.~Hu, and D.~Nualart.
	\newblock Nonlinear stochastic time-fractional slow and fast diffusion
	equations on {$\mathbb{R}^d$}.
	\newblock {\em Stochastic Process. Appl.}, 129(12):5073--5112, 2019.
	\MR{4025700}
	
	\bibitem[13]{CJKS}
	D.~Conus, M.~Joseph, D.~Khoshnevisan, and S.-Y. Shiu.
	\newblock On the chaotic character of the stochastic heat equation, {II}.
	\newblock {\em Probab. Theory Related Fields}, 156(3-4):483--533, 2013.
	\MR{3078278}
	
	\bibitem[14]{CM94}
	R.~A. Carmona and S.~A. Molchanov.
	\newblock Parabolic {A}nderson problem and intermittency.
	\newblock {\em Mem. Amer. Math. Soc.}, 108(518):viii+125, 1994.
	\MR{1185878}

    \bibitem[15]{con}
 D.~Conus, and D.~Khoshnevisan. 
 \newblock Weak nonmild solutions to some SPDEs. 
 \newblock {\em Illinois Journal of Mathematics} 54.4 (2010): 1329-1341.
 \MR{2981850}
 
	\bibitem[16]{corwin2012kardar}
	I.~Corwin.
	\newblock The {K}ardar{--P}arisi{--Z}hang equation and universality class.
	\newblock {\em Random Matrices: Theory Appl.}, 1(01):1130001, 2012.
	\MR{2930377}
	
	\bibitem[17]{d22}
	D.~Dauvergne.
	\newblock Non-uniqueness times for the maximizer of the {KPZ} fixed point.
	\newblock {\em Advances in Mathematics, 442}, p.109550, 2024.
	\MR{4708162}
 
	\bibitem[18]{dg}
	S.~Das and P.~Ghosal.
	\newblock Law of iterated logarithms and fractal properties of the {KPZ}
	equation.
	\newblock {\em The Annals of Probability}, 51(3), pp.930-986, 2023.
	\MR{4583059}
 
	\bibitem[19]{dgl}
	S.~Das, P.~Ghosal, and Y.~Lin.
	\newblock Long and short time laws of iterated logarithms for the {KPZ} fixed point.
	\newblock {\em arXiv preprint arXiv:2207.04162}, 2022.
	
	\bibitem[20]{dov}
	D.~Dauvergne, J.~Ortmann, and B.~Vir{\'a}g.
	\newblock The directed landscape.
	\newblock {\em Acta Mathematica}, 229(2), pp.201-285, 2022
	\MR{4035169}
	
	\bibitem[21]{dauvergne2020three}
	D.~Dauvergne, S.~Sarkar, and B.~Vir{\'a}g.
	\newblock Three-halves variation of geodesics in the directed landscape.
	\newblock {\em The Annals of Probability}, 50(5), pp.1947-1985, 2022
	\MR{4035169}
	
	\bibitem[22]{FK09}
	M.~Foondun and D.~Khoshnevisan.
	\newblock Intermittence and nonlinear parabolic stochastic partial differential
	equations.
	\newblock {\em Electron. J. Probab.}, 14:no. 21, 548--568, 2009.
	\MR{2480553}
	
	\bibitem[23]{foondun}
	M.~Foondun, D.~Khoshnevisan, and P.~Mahboubi.
	\newblock Analysis of the gradient of the solution to a stochastic heat
	equation via fractional {B}rownian motion.
	\newblock {\em Stochastic Partial Differential Equations: Analysis and
		Computations}, 3(2):133--158, 2015.
	\MR{3350450}
	
	\bibitem[24]{flo14}
	G.~R.~M. Flores.
	\newblock On the (strict) positivity of solutions of the stochastic heat
	equation.
	\newblock {\em The Annals of Probability}, 42(4):1635--1643, 2014.
	\MR{3262487}
	
	\bibitem[25]{g21}
	S.~Ganguly.
	\newblock Random metric geometries on the plane and {K}ardar-{P}arisi-{Z}hang
	universality.
	\newblock {\em arXiv preprint arXiv:2110.11287}, 2021.
	\MR{4353350}
	
	\bibitem[26]{bg21}
	S.~Ganguly and M.~Hegde.
	\newblock Local and global comparisons of the {A}iry difference profile to
	{B}rownian local time.
	\newblock {\em Annales de l'Institut Henri Poincare (B) Probabilites et statistiques}, vol. 59, no. 3, pp. 1342-1374. Institut Henri Poincaré, 2023.
	\MR{4635712}
	\bibitem[27]{gubinelli2015paracontrolled}
	M.~Gubinelli, P.~Imkeller, and N.~Perkowski.
	\newblock Paracontrolled distributions and singular {PDE}s.
	\newblock In {\em Forum of Mathematics, Pi}, volume~3. Cambridge University
	Press, 2015.
	\MR{3406823}
	
	\bibitem[28]{gonccalves2014nonlinear}
	P.~Gon{\c{c}}alves and M.~Jara.
	\newblock Nonlinear fluctuations of weakly asymmetric interacting particle
	systems.
	\newblock {\em Archive for Rational Mechanics and Analysis}, 212(2):597--644,
	2014.
	\MR{3176353}
	
	\bibitem[29]{yier}
	P.~Ghosal and Y.~Lin.
	\newblock Lyapunov exponents of the {SHE} for general initial data.
	\newblock {\em Annales de l'Institut Henri Poincare (B) Probabilites et statistiques}, vol. 59, no. 1, pp. 476-502. Institut Henri Poincaré, 2023.
	\MR{4533737}
 
	\bibitem[30]{GM90}
	J.~G\"{a}rtner and S.~A. Molchanov.
	\newblock Parabolic problems for the {A}nderson model. {I}. {I}ntermittency and
	related topics.
	\newblock {\em Comm. Math. Phys.}, 132(3):613--655, 1990.
	\MR{1069840}
	
	\bibitem[31]{bquart}
	J.~M. Guerra and D.~Nualart.
	\newblock The $1/h$-variation of the divergence integral with respect to the
	fractional {B}rownian motion for $h> 1/2$ and fractional {B}essel processes.
	\newblock {\em Stochastic processes and their applications}, 115(1):91--115,
	2005.
	\MR{2105371}
	
	\bibitem[32]{gubinelli2017kpz}
	M.~Gubinelli and N.~Perkowski.
	\newblock {KPZ} reloaded.
	\newblock {\em Communications in Mathematical Physics}, 349(1):165--269, 2017.
	\MR{3592748}
	
	\bibitem[33]{gubinelli2018energy}
	M.~Gubinelli and N.~Perkowski.
	\newblock Energy solutions of kpz are unique.
	\newblock {\em Journal of the American Mathematical Society}, 31(2):427--471,
	2018.
	\MR{3758149}
	
	\bibitem[34]{gy21}
	P.~Ghosal and J.~Yi.
	\newblock Fractal geometry of the valleys of the parabolic {A}nderson equation.
	\newblock {\em arXiv preprint arXiv:2108.03810}, 2021.
	
	\bibitem[35]{hairer2013solving}
	M.~Hairer.
	\newblock Solving the {KPZ} equation.
	\newblock {\em Annals of mathematics}, pages 559--664, 2013.
	\MR{3071506}
	
	\bibitem[36]{hairer2014theory}
	M.~Hairer.
	\newblock A theory of regularity structures.
	\newblock {\em Inventiones mathematicae}, 198(2):269--504, 2014.
	\MR{3274562}
	
	\bibitem[37]{HHNT}
	Y.~Hu, J.~Huang, D.~Nualart, and S.~Tindel.
	\newblock Stochastic heat equations with general multiplicative {G}aussian
	noises: {H}\"{o}lder continuity and intermittency.
	\newblock {\em Electron. J. Probab.}, 20:no. 55, 50, 2015.
	\MR{3354615}
	
	\bibitem[38]{hk}
	J.~Huang and D.~Khoshnevisan.
	\newblock On the multifractal local behavior of parabolic stochastic pdes.
	\newblock {\em Electronic Communications in Probability}, 22:1--11, 2017.
	\MR{3710805}
	
	\bibitem[39]{khoshnevisan2014analysis}
	D.~Khoshnevisan.
	\newblock {\em Analysis of stochastic partial differential equations}, volume
	119.
	\newblock American Mathematical Soc., 2014.
	\MR{3222416}
	
	\bibitem[40]{KKX17}
	D.~Khoshnevisan, K.~Kim, and Y.~Xiao.
	\newblock Intermittency and multifractality: a case study via parabolic
	stochastic {PDE}s.
	\newblock {\em Ann. Probab.}, 45(6A):3697--3751, 2017.
	\MR{3729613}
	
	\bibitem[41]{KKX18}
	D.~Khoshnevisan, K.~Kim, and Y.~Xiao.
	\newblock A macroscopic multifractal analysis of parabolic stochastic {PDE}s.
	\newblock {\em Comm. Math. Phys.}, 360(1):307--346, 2018.
	\MR{3795193}
	
	\bibitem[42]{khos}
	D.~Khoshnevisan, Y.~Peres, and Y.~Xiao.
	\newblock Limsup random fractals.
	\newblock {\em Electronic Journal of Probability}, 5, 2000.
	\MR{1743726}

 
	
	\bibitem[43]{kpz}
	M.~Kardar, G.~Parisi, and Y.-C. Zhang.
	\newblock Dynamic scaling of growing interfaces.
	\newblock {\em Physical Review Letters}, 56(9):889, 1986.
	
	\bibitem[44]{davar}
	D.~Khoshnevisan, J.~Swanson, Y.~Xiao, and L.~Zhang.
	\newblock Weak existence of a solution to a differential equation driven by a
	very rough fbm.
	\newblock {\em arXiv preprint arXiv:1309.3613}, 2013.
	
	\bibitem[45]{mqr}
	K.~Matetski, J.~Quastel, and D.~Remenik.
	\newblock The {KPZ} fixed point.
	\newblock {\em Acta Mathematica}, 227(1):115--203, 2021.
	\MR{4346267}
	
	\bibitem[46]{mueller1991support}
	C.~Mueller.
	\newblock On the support of solutions to the heat equation with noise.
	\newblock {\em Stochastics: An International Journal of Probability and
		Stochastic Processes}, 37(4):225--245, 1991.
	\MR{1149348}
	
	\bibitem[47]{blil}
	S.~Orey.
	\newblock Growth rate of certain {G}aussian processes.
	\newblock In {\em Proc. Sixth Berkeley Symp. Math. Statist. Probab}, pages
	443--451, 1972.
	\MR{0402897}
	
	\bibitem[48]{quastel2011local}
	J.~Quastel and D.~Remenik.
	\newblock Local {B}rownian property of the narrow wedge solution of the {KPZ}
	equation.
	\newblock {\em Electronic Communications in Probability}, 16:712--719, 2011.
	\MR{2861435}
	
	\bibitem[49]{qs20}
	J.~Quastel and S.~Sarkar.
	\newblock Convergence of exclusion processes and {KPZ} equation to the {KPZ} fixed
	point.
	\newblock {\em Journal of the American Mathematical Society} 36, no. 1 (2023): 251-289, 2023
	\MR{4495842}
	\bibitem[50]{quastel2011introduction}
	J.~Quastel.
	\newblock Introduction to {KPZ}.
	\newblock {\em Current developments in mathematics}, 2011(1), 2011.
	\MR{3098078}
	
	\bibitem[51]{bq1}
	L.~C.~G. Rogers.
	\newblock Arbitrage with fractional {B}rownian motion.
	\newblock {\em Mathematical Finance}, 7(1):95--105, 1997.
	\MR{1434408}
	
	\bibitem[52]{vir20}
	B.~Vir{\'a}g.
	\newblock The heat and the landscape {I}.
	\newblock {\em arXiv preprint arXiv:2008.07241}, 2020.
	
	\bibitem[53]{walsh1986}
	J.~B. Walsh.
	\newblock An introduction to stochastic partial differential equations.
	\newblock In {\em {\'E}cole d'{\'E}t{\'e} de Probabilit{\'e}s de Saint Flour
		XIV-1984}, pages 265--439. Springer, 1986.
	\MR{0876085}
	
\end{thebibliography}




\end{document}